\documentclass[11pt]{article}
\usepackage[T1]{fontenc}
\usepackage[utf8]{inputenc}
\usepackage{lmodern}
\usepackage{amsmath,amssymb,amsthm,mathtools}
\usepackage{mathrsfs}
\usepackage[margin=1in]{geometry}
\usepackage[numbers,sort&compress]{natbib}
\usepackage[hidelinks]{hyperref}
\newtheorem{theorem}{Theorem}
\newtheorem{lemma}{Lemma}

\theoremstyle{definition}
\newtheorem{definition}{Definition}
\theoremstyle{remark}
\newtheorem{remark}{Remark}
\newcommand{\E}{\mathbb E}
\newcommand{\R}{\mathbb R}
\newcommand{\GGC}{\operatorname{GGC}}
\newcommand{\DP}{\operatorname{DP}}
\newcommand{\Law}{\mathcal L}
\newcommand{\dd}{\mathop{}\!\mathrm d}
\newcommand{\Ptwo}{\mathcal P_2(\R)}

\title{Powers of generalized gamma convolutions}
\author{Min Wang\thanks{School of Mathematics and Statistics, Wuhan University of Technology, Wuhan, 430063, China.
Email: \texttt{minwangmath@whut.edu.cn}.}
\and Sheng Yin\thanks{Institute for Advanced Study in Mathematics, Harbin Institute of Technology, Harbin, 150001, China.
Email: \texttt{sheng.yin@hit.edu.cn}.}}
\date{}

\begin{document}
\maketitle

\begin{abstract}
In 2015, Bondesson showed that Thorin’s class of generalized gamma convolutions (GGCs) has the remarkable property of being closed under multiplication of independent random variables. He also conjectured that the GGC class is closed under taking powers of order greater than one. In this paper,  we provide a candidate proof for Bondesson’s conjecture. %The proof uses a nonlinear evolution of the Thorin measure in logarithmic rate coordinates. A gamma--Dirichlet representation and a bounded Stieltjes phase yield a nonnegative jump kernel and a drift of at most linear growth. These estimates allow us to construct the evolution by positive Euler approximations and to identify its associated distributions through a linear transport equation. Weak approximation by finite gamma convolutions then gives the result without restrictions on the drift, the Thorin mass, or the moments of the initial distribution.
\end{abstract}

\section{Introduction}
\label{sec:introduction}

In this paper, we study the behavior of generalized gamma convolutions under powers of order greater than one.
We use the following definition%, which includes degenerate distributions
; see \cite[Section~3.1]{Bondesson1992}.

\begin{definition}[Generalized gamma convolutions]
\label{def:ggc}
A probability distribution on $[0,\infty)$ is a generalized gamma convolution if it is a weak limit of distributions of finite sums of independent gamma random variables.
We denote this class by $\GGC$.
\end{definition}

The class is closed with
respect to change in scale, weak limits, and addition of independent random variables.
Bondesson proved that it is also closed under multiplication of independent random variables \cite[Theorem~1]{Bondesson2015}.
In \cite[Conjecture~1, p.~1075]{Bondesson2015}, Bondesson asked whether $X^q$ has a GGC distribution whenever $X$ does and $q\ge1$.
Here and throughout, $X^q$ is the ordinary power of a nonnegative random variable.

Our main result answers this question.

\begin{theorem}[Power closure]
\label{thm:main}
Let $X$ be a nonnegative real-valued random variable with a generalized gamma convolution distribution.
Then $X^q$ has a generalized gamma convolution distribution for every real $q\ge1$.
\end{theorem}

%The statement allows positive drift, infinite Thorin mass, and arbitrary tails.It also includes the constant zero.The main part of the proof treats finite gamma convolutions; weak closure will then give the full assertion.

Let $G_1,\ldots,G_n$ be independent gamma random variables with $G_i\sim\operatorname{Gamma}(\beta_i,1)$. Here $\operatorname{Gamma}(\beta,b)$ denotes the gamma distribution with shape $\beta>0$ and rate $b>0$.
Our initial variable has the form
\begin{equation}
 X=\sum_{i=1}^n b_i^{-1}G_i,\qquad
 n\ge1,\quad b_i,\beta_i>0,\qquad B_0=\sum_{i=1}^n\beta_i,
 \label{eq:finite-input}
\end{equation}
where each summand $b_i^{-1}G_i$ has distribution $\operatorname{Gamma}(\beta_i,b_i)$.
Its Thorin measure is $U_0=\sum_i\beta_i\delta_{b_i}$.
Writing
\[
 F_0=B_0^{-1}\sum_{i=1}^n\beta_i\delta_{\log b_i},\qquad
 B_t=B_0e^{-t},
\]
we seek a curve of probability measures $F_t$ for which $U_t=B_t\exp_*F_t$ is the Thorin measure of $X^{e^t}$.
Here $\exp_*F_t$ is the push-forward of $F_t$ under $y\mapsto e^y$, converting log rates to rates; multiplication by $B_t$ gives the total Thorin mass.
The measure $F_t$ describes logarithmic rates; the law of $\log(X^{e^t})$ is a different measure.

The first difficulty is to express power differentiation in a form that preserves positivity of the Thorin measure.
We use the gamma--Dirichlet representation of an exponentially tilted GGC variable to compute this derivative.
The posterior identity for a Dirichlet process then expresses it through a bounded Stieltjes phase.
In logarithmic rate coordinates, this gives an operator with a nonnegative jump kernel, a finite jump second moment, and a drift whose difference from $y$ is bounded independently of the rate distribution.
These bounds are the basis of the existence argument.

The second difficulty is to identify the distributions obtained from that evolution.
The agreement of the infinitesimal derivatives alone does not make this identification.
We prove that the associated value distributions satisfy the weak equation with operator $x\log x\,\partial_x$.
Taking logarithms reduces this equation to linear transport, whose solution is dilation by $e^t$.
It follows that the constructed distributions are precisely the laws of $X^{e^t}$.

Section~\ref{sec:foundations} recalls the Thorin and Dirichlet representations and the bounded phase formula.
Section~\ref{sec:tangent} derives the positive operator, and Section~\ref{sec:evolution} constructs its evolution.
The identification is proved in Section~\ref{sec:identification}, and the weak approximation argument in Section~\ref{sec:completion} completes the proof of Theorem~\ref{thm:main}.
Appendix~\ref{sec:measurable-realizations} supplies the measurable realizations used in the coefficient construction and its continuity proof.

\paragraph{Notation.}
We write $\Law(X)$ for the law of $X$, $\E$ for expectation, and $\overset{\mathrm d}=$ for equality in distribution.
For a measure $\nu$, write $\nu(f)=\int f\,\dd\nu$ whenever the integral is defined, and let $T_*\nu$ be its image under a measurable map $T$.
Weak convergence of probability measures is denoted by $\Rightarrow$.
Further notation is introduced as it is needed.

\paragraph{Disclosure of AI assistance.}
We used OpenAI's models 5.6~\texttt{Sol} and 6~\texttt{Astra} extensively in conducting this research and preparing the manuscript.
Our use of these systems included literature searches and source checks, the organization of research notes, the comparison and revision of proof strategies, the development of mathematical constructions and intermediate arguments, the computational exploration, and the manuscript drafting.
In particular, we drew on AI assistance in developing and internally checking the log-rate generator, the positive Euler evolution, and the identification argument.
We disclose this use as substantive mathematical assistance, and our exploration is documented at the GitHub repository (\url{https://github.com/vtejdn/generalized-gamma-convolution-power-problem}).
See the repository for the complete record of our AI-assisted exploration and development.
The material in this repository is released under the Apache License 2.0, which permits reuse, modification and distribution.

\paragraph{Status of this manuscript.}
We are in the process of finalizing the manuscript, including incorporating feedback and ensuring all proofs and arguments are thoroughly checked.
We share the manuscript and our exploration record in their current form to invite feedback and collaboration from the interested readers.
We will update the repository and the manuscript as we make progress and incorporate feedback.
We will take responsibility for ensuring that the final version of the manuscript is accurate, complete with respect to all references, proofs, and arguments under common standards of the research community.

\section{Preliminaries}
\label{sec:foundations}

We collect the probabilistic and analytic identities used in the proof.
The Thorin representation reduces the problem to positive measures on the rate variable, while the gamma--Dirichlet and posterior identities will allow us to differentiate powers of the corresponding random variables.

For $\beta,b>0$, $\operatorname{Gamma}(\beta,b)$ denotes the gamma distribution with shape $\beta$ and rate $b$, having density
\[
 \frac{b^\beta}{\Gamma(\beta)}x^{\beta-1}e^{-bx},\qquad x>0,
 \qquad
 \Gamma(\beta)=\int_0^\infty x^{\beta-1}e^{-x}\,\dd x.
\]
We specify the shape of each gamma random variable through its distribution: if $G\sim\operatorname{Gamma}(\beta,1)$, then $G/b\sim\operatorname{Gamma}(\beta,b)$. Subscripts on gamma random variables label the variables, rather than their shape parameters.
The symbol $\perp$ denotes independence.

We use the integration notation introduced above also for unbounded integrable tests.
Expressions such as $F(y^2)$ and $F(|y|)$ stand for $\int y^2F(\dd y)$ and $\int |y|F(\dd y)$.
The point mass at $x$ is denoted by $\delta_x$, and $\mathbf1_A$ denotes the indicator of a set $A$.
Thus $T_*\nu(A)=\nu(T^{-1}(A))$ for a measurable map $T$.
We use $L_\nu(s)=\int e^{-sx}\nu(\dd x)$ for the Laplace transform of a probability on $[0,\infty)$; this is distinct from the law notation $\Law(X)$.

We write $\mathcal P(\R)$ for the Borel probability measures on $\R$ and $\Ptwo$ for those with finite second moment.
Weak, or narrow, convergence is denoted by $\Rightarrow$ and means convergence against bounded continuous functions.
Curves with values in $\Ptwo$ will be continuous for the narrow topology unless otherwise stated.
Unless specified otherwise, $L^p$ spaces on real intervals use Lebesgue measure.
We write $C_c^k(D)$ for the $k$-times continuously differentiable real functions with compact support in the open set $D$, and $\|f\|_\infty$ for the supremum norm on the indicated domain.
Finally, $r_+=r^+=\max\{r,0\}$, $r_-=r^-=\max\{-r,0\}$, and $\log_+x=\max\{\log x,0\}$ for $x>0$.

We also use the digamma function
\[
  \psi_0(r)=\frac{\Gamma'(r)}{\Gamma(r)},\qquad r>0.
\]

\subsection{Thorin measures}

A probability law \(\rho\) on \([0,\infty)\) belongs to \(\GGC\) if and only if its Laplace transform has the representation
\begin{equation}\label{eq:thorin-representation}
  L_\rho(s):=\int_{[0,\infty)}e^{-sx}\rho(\dd x)
  =\exp\left\{-as-\int_{(0,\infty)}
                  \log(1+s/b)\,U(\dd b)\right\},\qquad s\ge0,
\end{equation}
where \(a\ge0\) and \(U\) is a positive Borel measure satisfying
\begin{equation}\label{eq:thorin-admissibility}
  \int_{(0,\infty)}\log(1+1/b)\,U(\dd b)<\infty.
\end{equation}
The measure \(U\) is the rate-form Thorin measure and \(a\) is the drift.
We will first work with zero drift and finite Thorin mass; these are separate restrictions.
This is the Laplace-transform form of the representation in \cite[Section~3.1, p.~29, equations~(3.1.1)--(3.1.2)]{Bondesson1992}.
Condition \eqref{eq:thorin-admissibility} combines the usual integrability conditions at zero and infinity.
It also implies that \(U\) is finite on compact subintervals of \((0,\infty)\).

\begin{lemma}[Weak closure and finite-gamma approximation]
\label{lem:ggc-closure}
The class \(\GGC\) is closed under weak limits that are probability laws.
Every law in \(\GGC\) is the weak limit of laws of the form
\begin{equation}\label{eq:finite-gamma-approximation}
  \Law\left(\sum_{j=1}^{m_n}\frac{G_{n,j}}{b_{n,j}}\right),
  \qquad \beta_{n,j}>0,\quad b_{n,j}>0,
\end{equation}
where, for each $n$, the variables $G_{n,j}\sim\operatorname{Gamma}(\beta_{n,j},1)$, $1\le j\le m_n$, are independent.
Thus the approximating laws may be chosen with zero drift and finite atomic Thorin measures, even when the limiting law has positive drift or an infinite Thorin measure.
\end{lemma}

\begin{proof}
The weak closure assertion is \cite[Theorem~3.1.5, pp.~34--35]{Bondesson1992}, applied to a limit that is a probability measure.
The last paragraph on p.~35 of the same reference states the approximation assertion explicitly, with approximating drifts zero and approximating Thorin measures having finitely many atoms.
The transform of a finite atomic measure \(U_n=\sum_j\beta_{n,j}\delta_{b_{n,j}}\) is the transform of the sum in \eqref{eq:finite-gamma-approximation}.
If the target law is \(\delta_0\), a single gamma variable with a rate tending to infinity also supplies the asserted approximation.
\end{proof}

For a law with representation \eqref{eq:thorin-representation}, differentiation at \(s>0\) gives
\begin{equation}\label{eq:ggc-log-derivative}
  -\partial_s\log L_\rho(s)
  =a+\int_{(0,\infty)}\frac{U(\dd b)}{s+b}.
\end{equation}
The differentiation follows from \eqref{eq:thorin-admissibility}, locally uniformly in positive \(s\).

\subsection{Dirichlet processes and gamma means}

The following convention for Dirichlet processes is the one used in \cite[Section~1, reprint p.~2]{James2005}.

\begin{definition}[Dirichlet processes]
\label{def:dirichlet-process}
For a nonzero finite measure \(V\) on a Borel space \(E\), \(P\sim\DP(V)\) means that \(P\) is a \emph{Dirichlet process}: it is a random probability measure whose masses on any measurable finite partition of \(E\) have the Dirichlet distribution with parameters equal to the corresponding \(V\)-masses.
For positive parameters \((\beta_1,\ldots,\beta_m)\), this Dirichlet distribution is the law of \((G_i/\sum_{j=1}^mG_j)_{i=1}^m\), where \(G_1,\ldots,G_m\) are independent and \(G_i\sim\operatorname{Gamma}(\beta_i,1)\).
A cell of zero base mass has zero random mass almost surely and is omitted from this vector.
We use this notation both for rates, with \(E=(0,\infty)\), and for log rates, with \(E=\R\).
The base measure may be atomic, nonatomic, or mixed.
\end{definition}

\begin{lemma}[Gamma--Dirichlet normalization and Markov--Krein identity]
\label{lem:gamma-dirichlet}
Let \(U\) be a nonzero finite positive measure on \((0,\infty)\), with total mass \(B\).
Let \(P\sim\DP(U)\), and let \(G\sim\operatorname{Gamma}(B,1)\) be independent of \(P\).
Then the random measure \(GP\) is the gamma process with shape measure \(U\) and unit rate: its masses on disjoint sets are independent gamma variables with shapes equal to their \(U\)-masses, and sets of zero \(U\)-mass receive zero mass.
In particular, if \(r\ge0\) is Borel measurable and
\[
  \int\log(1+r(b))\,U(\dd b)<\infty,
\]
then \(P(r):=\int r(b)P(\dd b)<\infty\) almost surely, and
\begin{equation}\label{eq:markov-krein}
  \E e^{-zGP(r)}
  =\E(1+zP(r))^{-B}
  =\exp\left\{-\int\log(1+zr(b))\,U(\dd b)\right\},
  \qquad z\ge0.
\end{equation}
Consequently, if \(U\) satisfies \eqref{eq:thorin-admissibility}, the zero-drift GGC \(X\) with Thorin measure \(U\) admits the representation
\begin{equation}\label{eq:gamma-dirichlet-untilted}
  X\ \overset{\mathrm d}=\ G\int b^{-1}P(\dd b).
\end{equation}
For \(s>0\), let \(X_s\) denote the exponential tilt of \(X\):
\[
  \Law(X_s)(\dd x)
  =\frac{e^{-sx}}{\E e^{-sX}}\,\Law(X)(\dd x).
\]
Then, with
\[
  M_P(s):=\int_{(0,\infty)}\frac{P(\dd b)}{s+b},
\]
one has
\begin{equation}\label{eq:gamma-dirichlet-tilt}
  X_s\ \overset{\mathrm d}=\ GM_P(s),\qquad G\ \text{independent of }P.
\end{equation}
The base law of \(P\) on the right is \(\DP(U)\) for every \(s\).
\end{lemma}

\begin{proof}
The normalization, independence, logarithmic condition, and identity \eqref{eq:markov-krein} are recorded in \cite[Section~1, reprint p.~2, equations~(1)--(3)]{James2005}.
We recall the finite-partition argument to make explicit its extension to the unbounded functions used here.
For a finite partition with positive parameter vector \((\beta_1,\ldots,\beta_m)\), the change of variables \(x_j=tp_j\), with Jacobian \(t^{m-1}\), factors the joint density of independent unit-rate gamma variables into a \(\operatorname{Gamma}(B,1)\) density in \(t\) and a Dirichlet density in \((p_1,\ldots,p_m)\).
Conversely, multiplying a Dirichlet vector independent of \(G\) by \(G\) gives these independent gamma variables.
The gamma-process assertion follows by applying this calculation to every finite partition.

The Laplace functional for a nonnegative simple function is therefore the right side of \eqref{eq:markov-krein}.
Increasing simple approximation proves the same identity for an arbitrary nonnegative Borel function, initially allowing an infinite integral.
Under the stated logarithmic condition, the right side tends to one as \(z\downarrow0\), by dominated convergence with \(0\le z\le1\).
It follows that \(GP(r)\), and hence \(P(r)\), is finite almost surely.
Conditioning on \(P\) and integrating the gamma density gives the middle expression in \eqref{eq:markov-krein}.

Taking \(r(b)=b^{-1}\) proves \eqref{eq:gamma-dirichlet-untilted}.
For the tilted identity, \(r(b)=(s+b)^{-1}\) is bounded and
\begin{align*}
  \E e^{-vGM_P(s)}
  &=\exp\left\{-\int
       \log\left(1+\frac{v}{s+b}\right)U(\dd b)\right\}\\
  &=\frac{\E e^{-(s+v)X}}{\E e^{-sX}},\qquad v\ge0.
\end{align*}
Uniqueness of Laplace transforms proves \eqref{eq:gamma-dirichlet-tilt}.
The same auxiliary pair \((G,P)\) can thus represent each tilted marginal; only these marginal identities will be used.
\end{proof}

For later use, \(\operatorname{Beta}(1,B)\) denotes the distribution on \((0,1)\) with density \(B(1-z)^{B-1}\).
Sampling one point from a Dirichlet process adds a unit atom to its base measure in the posterior law.
The following form will turn random rate integrals into deterministic ones.

\begin{lemma}[One-observation posterior and Palm identity]
\label{lem:dirichlet-posterior}
Let \(U\) be a nonzero finite positive measure on \((0,\infty)\) with total mass \(B\).
Write \(\E_{\DP(V)}\) for expectation over a random probability measure \(P\sim\DP(V)\).
For every nonnegative jointly Borel function \(\Phi(b,P)\),
\begin{equation}\label{eq:dirichlet-palm}
  \E_{\DP(U)}\int\Phi(b,P)P(\dd b)
  =\frac1B\int U(\dd b)\,
          \E_{\DP(U+\delta_b)}\Phi(b,P).
\end{equation}
The identity also holds for a signed jointly Borel function whenever \(\E_{\DP(U)}\int|\Phi(b,P)|P(\dd b)<\infty\).
A jointly measurable realization of the posterior laws is
\begin{equation}\label{eq:palm-coupling}
  P^{(b)}=(1-Z)Q+Z\delta_b,\qquad
  Q\sim\DP(U),\quad Z\sim\operatorname{Beta}(1,B),
  \qquad Q\ \text{independent of }Z,
\end{equation}
for which \(P^{(b)}\sim\DP(U+\delta_b)\).
\end{lemma}

\begin{proof}
The posterior \(P\mid Y=b\), when \(Y\mid P\) has law \(P\), is \(\DP(U+\delta_b)\); see \cite[Section~2, reprint pp.~4--5, including equation~(8)]{James2005}.
We give a proof using finite partitions.
Multiplication of a Dirichlet density with parameter vector \((\beta_j)_j\) by its \(i\)-th coordinate changes the parameter to \((\beta_j+\mathbf1_{\{j=i\}})_j\), with multiplier \(\beta_i/B\).
The assertion is trivial for a partition with one positive-mass cell.
On a common finite partition this proves \eqref{eq:dirichlet-palm} for cylinder functions and simple functions of the sampled location.
Refinement and the monotone-class theorem extend the identity to the stated jointly Borel functions.
Zero-mass cells contribute nothing.
This proof does not assume that \(U\) is diffuse.

To obtain \eqref{eq:palm-coupling}, let \(G\sim\operatorname{Gamma}(B,1)\) be independent of \(Q\). Add a \(\operatorname{Gamma}(1,1)\) mass independent of \((G,Q)\) at \(b\) to the gamma process \(GQ\), and normalize its total mass.
The new shape measure is \(U+\delta_b\).
The fraction assigned by the added mass is \(\operatorname{Beta}(1,B)\), independently of \(Q\).
The displayed formula, as a function of \((b,Q,Z)\), is a Borel map into the space of probability measures endowed with its weak Borel sigma-field.
It therefore also specifies the required measurable posterior kernel.
\end{proof}

Atomic base measures are allowed throughout, so the identity applies in particular to finite gamma convolutions.

\begin{lemma}[Stick-breaking realization]\label{lem:stick-breaking}
Let \(B>0\), let \(H\) be a probability on \(\R\), let \((Y_j)_{j\ge1}\) be independent with law \(H\), and let \((V_j)_{j\ge1}\) be independent \(\operatorname{Beta}(1,B)\) variables, independently of the locations.
Then
\begin{equation}\label{eq:stick-breaking}
  Q=\sum_{j\ge1}W_j\delta_{Y_j},\qquad
  W_j=V_j\prod_{i<j}(1-V_i),
\end{equation}
is a probability measure almost surely and has law \(\DP(BH)\).
\end{lemma}

\begin{proof}
This is Sethuraman's construction \cite[Section~2, pp.~642--643, equation~(2.1), and Theorem~3.4, p.~645]{Sethuraman1994}.
Note that the remaining mass after \(m\) terms is \(\prod_{j\le m}(1-V_j)\); it decreases to zero almost surely because its expectation is \((B/(B+1))^m\).

For a finite partition with probabilities \(p_i\), a Dirichlet vector \(D\) with parameters \(Bp_i\) satisfies
\[
  D\ \overset{\mathrm d}=\ V e_J+(1-V)D',
\]
where \(e_i\) is the \(i\)-th coordinate unit vector, \(D'\) has the same law as \(D\), \(V\sim\operatorname{Beta}(1,B)\), and \(\mathbb P(J=i)=p_i\); the variables \(D',V,J\) are mutually independent.
Indeed, conditional on \(J=i\), gamma addition gives the Dirichlet parameters \(Bp_j+\mathbf1_{\{j=i\}}\).
Its density relative to that of \(D\) is \(x_i/p_i\); mixing with weights \(p_i\) restores the original density.
Again zero-probability cells are omitted and the one-cell case is immediate.
This fixed point is unique: two initial probability vectors driven by the same \(V_j,J_j\) have \(\ell^1\) distance at most \(2\prod_{j\le m}(1-V_j)\) after \(m\) iterations.
The partition vector defined by \eqref{eq:stick-breaking} satisfies this fixed-point equation and hence is Dirichlet.
This proves the lemma on all finite partitions, including for atomic \(H\).
\end{proof}

When the parameters vary, these random measures can be realized on a single probability space.
The precise construction is given in Lemma~\ref{lem:parameterized-dirichlet} of Appendix~\ref{sec:measurable-realizations}.
It will be used both for measurability of the coefficients and for their weak continuity.

\subsection{The bounded Stieltjes phase}

The phase representation below is the analytic input that makes the jump kernel nonnegative.
We use the convention for Stieltjes functions in \cite[Chapter~2]{SSV2010}.

\begin{definition}[Stieltjes functions]
\label{def:stieltjes}
A Stieltjes function is a function \(S:(0,\infty)\to[0,\infty)\) with a representation
\[
  S(s)=\frac{c}{s}+d+\int_{(0,\infty)}
                         \frac{\sigma(\dd t)}{s+t},
  \qquad c,d\ge0,\qquad
  \int\frac{\sigma(\dd t)}{1+t}<\infty,
\]
where \(\sigma\) is a positive Borel measure.
\end{definition}

The bounded phase representation used below is a consequence of two results in \citet[Theorems~6.10 and~7.3]{SSV2010}: if \(S\) is a nonzero Stieltjes function, its reciprocal \(f=1/S\) has the representation
\begin{equation}\label{eq:ssv-exponential}
  f(s)=\exp\left\{c_0+\int_0^\infty
     \left(\frac{t}{1+t^2}-\frac1{s+t}\right)
                           \xi(t)\,\dd t\right\},
  \qquad 0\le\xi\le1,
\end{equation}
with \(c_0\in\R\) and \(\xi\) unique up to Lebesgue-null sets \cite[Theorem~6.10, pp.~58--59, and Theorem~7.3, p.~63]{SSV2010}.
In the terminology of that reference, \(1/S\) is a complete Bernstein function.
We need only the displayed representation and its uniqueness.
In particular, they apply to the reciprocal of every resolvent mean considered below.

\begin{lemma}[Measurable phase and anchor-one representation]
\label{lem:phase}
For every probability measure \(P\) on \((0,\infty)\), extend \(M_P\) to the slit plane by
\[
  M_P(z)=\int\frac{P(\dd b)}{z+b},
  \qquad z\in\mathbb C\setminus(-\infty,0].
\]
There is a jointly Borel function \((P,t)\mapsto\xi_P(t)\in[0,1]\) such that
\begin{equation}\label{eq:phase-anchor}
  \log M_P(s)-\log M_P(1)
  =\int_0^\infty\xi_P(t)
       \left(\frac1{s+t}-\frac1{1+t}\right)\dd t,
  \qquad s>0.
\end{equation}
Here the probability measures \(P\) carry the Borel sigma-field of the narrow topology.
For each fixed \(P\), the density in this representation is unique Lebesgue-almost everywhere.
A specific jointly Borel choice is
\begin{equation}\label{eq:phase-representative}
  \xi_P(t)=\limsup_{n\to\infty}
       \frac{\arg(1/M_P(-t+i/n))}{\pi}
       =\limsup_{n\to\infty}
       \frac{-\arg M_P(-t+i/n)}{\pi},\qquad t>0,
\end{equation}
where \(n\) runs through the positive integers and \(\arg\) is the principal argument.
The ordinary limit in \eqref{eq:phase-representative} exists for almost every \(t\).
Both \eqref{eq:phase-anchor} and all its positive-order \(s\)-derivatives are absolutely convergent, locally uniformly in \(s>0\), uniformly over \(P\).
\end{lemma}

The representation follows by subtracting the logarithmic formula \eqref{eq:ssv-exponential} at $s$ and at $1$.
The measurable boundary version and the absolute convergence assertions are proved in Appendix~\ref{sec:phase-proof}.

Normalization at \(1\) is useful here because it requires neither \(M_P(0+)<\infty\) nor a logarithmic moment of the large rates.
The extra assumption for the zero-anchored variant in \cite[Remark~6.11, p.~60]{SSV2010} is therefore not used.
All references to \cite{SSV2010} use the 2010 first edition with the authors' posted corrections.

\section{Power differentiation and a positive generator}
\label{sec:tangent}

We first differentiate the power parameter and then express the result as the action of a positive operator on logarithmic rates.
The gamma--Dirichlet representation supplies the derivative; the posterior identity and the bounded phase formula supply the operator.

\subsection{Differentiating the power parameter}

Throughout this subsection $U$ is a positive measure on $(0,\infty)$ with
\begin{equation}
 0<B=U((0,\infty))<\infty,\qquad
 \int\log(1+b^{-1})\,U(\dd b)<\infty.
 \label{eq:finite-admissible}
\end{equation}
Let $X$ have the zero-drift GGC law with Thorin measure $U$.
Thus
\[
 L(s)=\E e^{-sX}
 =\exp\left\{-\int\log(1+s/b)\,U(\dd b)\right\}.
\]
For $s>0$, let $X_s$ have the exponential tilt of $X$ defined in Lemma~\ref{lem:gamma-dirichlet}.
Recall the digamma function $\psi_0=\Gamma'/\Gamma$.
For a probability measure $P$ on $(0,\infty)$, put
\begin{equation}
 M_P(s)=\int\frac{P(\dd b)}{s+b},\qquad
 W_P(s)=M_P(s)+sM_P'(s)=\int\frac{b\,P(\dd b)}{(s+b)^2}.
 \label{eq:MW}
\end{equation}

\begin{lemma}[The power derivative]
\label{lem:tangent}
Let $U$, $B$, and $X$ be as above.
Let $P\sim\DP(U)$ and $G\sim\operatorname{Gamma}(B,1)$ be independent.
Then, for $s>0$,
\begin{equation}
 X_s\overset d=GM_P(s).
 \label{eq:tilt-gamma-dirichlet}
\end{equation}
For $s,q>0$, define
\[
 g_q(s)=-\partial_s\log\E e^{-sX^q},\qquad
 h_U(s)=\left.\partial_q g_q(s)\right|_{q=1},\qquad g(s)=g_1(s).
\]
The derivatives defining $g_q$ and $h_U$ exist, and expectations on the right-hand sides below are with respect to $P\sim\DP(U)$:
\begin{align}
 g(s)&=B\E M_P(s),\label{eq:g-dirichlet}\\
 h_U(s)&=B\E\left[
 W_P(s)\{\psi_0(B+1)+\log M_P(s)\}+sM_P'(s)
 \right],\label{eq:tangent}\\
 \frac{h_U(s)+g(s)}B
 &=\E\left[W_P(s)\{\psi_0(B+1)+1+\log M_P(s)\}\right].
 \label{eq:normalized-tangent}
\end{align}
Equivalently, $h_U(s)=\partial_s\{s\E[X_s\log X_s]\}$.
\end{lemma}

\begin{proof}
The tilted representation \eqref{eq:tilt-gamma-dirichlet} is \eqref{eq:gamma-dirichlet-tilt}.
We use it as an equality of distributions for each fixed $s$.

For $q$ in a compact subinterval of $(0,\infty)$, and $s$ in a compact positive interval, differentiating $e^{-s x^q}$ in $s$ and $q$ produces finite sums of bounded functions of $z=x^q$ of the form $z^r(\log z)^j e^{-sz}$ with $r>0$, multiplied by bounded powers of $q^{-1}$.
At $x=0$ these terms have the continuous value zero.
Differentiation under the original probability is therefore justified without a moment assumption on $X$.
Differentiating first in $q$ yields
\[
 \left.\partial_q\log\E e^{-sX^q}\right|_{q=1}
 =-s\E[X_s\log X_s],
\]
and hence $h_U(s)=\partial_s\{s\E[X_s\log X_s]\}$.

To evaluate the tilted moment, we differentiate the gamma integral with respect to its positive shape parameter and obtain $\E[G\log G]=B\psi_0(B+1)$.
Consequently
\[
 \E[X_s\log X_s]
 =B\E\left[M_P(s)\{\psi_0(B+1)+\log M_P(s)\}\right].
\]
It remains to differentiate this expression in $s$.
For every $P$ and integer $j\ge1$,
\[
 0<W_P(s)\leq M_P(s)\leq s^{-1},\qquad
 |M_P^{(j)}(s)|\leq j!s^{-j}M_P(s).
\]
On compact positive $s$-intervals, the needed derivatives of $M_P(s)\log M_P(s)$ are bounded by a constant times $M_P(s)(1+|\log M_P(s)|)$, uniformly bounded for $0<M_P(s)\leq s^{-1}$.
Thus we may differentiate this last expectation, which gives \eqref{eq:tangent}.
Equation~\eqref{eq:g-dirichlet} follows either from the Thorin representation or from the mean of the tilted law.
Adding it to \eqref{eq:tangent} and using $W_P=M_P+sM_P'$ proves \eqref{eq:normalized-tangent}.
\end{proof}

The lemma applies to any current zero-drift GGC law with a finite Thorin measure.
This is what will allow us to use the same identity at every time in the evolution.

\subsection{The operator on logarithmic rates}
\label{sec:generator}

We now represent the right-hand side of \eqref{eq:normalized-tangent} by an operator on the logarithmic rate variable.
The construction will apply to every base probability, and its estimates will be independent of the support of that probability.

Let $B>0$ and let $F$ be any probability on $\R$.
Set $U=B\exp_*F$.
The coefficients below are defined even when $U$ fails the Thorin integrability condition; admissibility will follow from the moment bound along the evolution.
For $b=e^y$ use the posterior probability
\[
 P^{(b)}=(1-Z)Q+Z\delta_b,\qquad
 Q\sim\DP(U),\quad Z\sim\operatorname{Beta}(1,B),\quad Q\perp Z,
\]
from \eqref{eq:palm-coupling}.
For a measurable function $A$ of a probability measure, write
\[
 \E^{(b)}A(P):=\E[A(P^{(b)})]
             =\E_{\DP(U+\delta_b)}A(P)
\]
whenever the expectation exists.
Let $\xi_P$ be the jointly measurable phase in Lemma~\ref{lem:phase}.
The jump measure and drift are obtained from the following four quantities:
\begin{align}
 \nu_0(\dd v)&=\frac{e^v}{(e^v-1)^2}\,\dd v
             =\frac{\dd v}{4\sinh^2(v/2)},\qquad v\ne0,
             \label{eq:universal-jumps}\\
 K(u)&=\frac1{1+u}-\frac1{u-1}+\frac{\log u}{(u-1)^2},
       \quad u\ne1,\qquad K(1)=0,\label{eq:K}\\
 k_{B,F}(y,v)&=\E^{(b)}\xi_P(be^v),\label{eq:acceptance}\\
 a_{B,F}(y)&=y-\psi_0(B+1)-1
 +\E^{(b)}\left[-\log\{bM_P(b)\}
                +\int_0^\infty\xi_P(bu)K(u)\,\dd u\right].
 \label{eq:log-drift}
\end{align}
Here $u>0$, $y,v\in\R$, and $b=e^y$; the measure $\nu_0$ has no atom at zero.
The function $K$ is a correction term that will cancel the logarithmic compensation on resolvent tests.
Its three terms are integrated together, and its singularity at one is removable.
The bound $0\le\xi_P\le1$ makes $k_{B,F}$ an acceptance factor for the reference measure $\nu_0$.
Although this measure has infinite mass near zero, its second moment is finite.

\begin{lemma}[Uniform coefficient bounds]
\label{lem:bounds}
The coefficients are finite and Borel measurable, with $0\leq k_{B,F}\leq1$.
Moreover
\begin{align}
 \int_0^1 K(u)\,\dd u&=-(1-\log2),&
 \int_1^\infty K(u)\,\dd u&=1-\log2,\label{eq:K-integrals}\\
 \int_{\R}v^2\nu_0(\dd v)&=4\sum_{n=1}^\infty n^{-2}
 =:m_2<\infty,\label{eq:m2}\\
 y-\psi_0(B+1)-2+\log2
 &\leq a_{B,F}(y)\leq y-\psi_0(1).
 \label{eq:drift-bounds}
\end{align}
In particular $|a_{B,F}(y)-y|\leq C_B$, where $C_B$ can be chosen uniformly bounded on compact subintervals of $0<B<\infty$, independently of $F$.

For every twice continuously differentiable $\varphi$ with bounded second derivative, define the absolutely convergent integral
\begin{equation}
 \mathcal G_{B,F}\varphi(y)
 =a_{B,F}(y)\varphi'(y)
 +\int_{\R}\{
   \varphi(y+v)-\varphi(y)-v\varphi'(y)
   \}k_{B,F}(y,v)\nu_0(\dd v).
 \label{eq:generator}
\end{equation}
It satisfies
\begin{equation}
 |\mathcal G_{B,F}\varphi(y)|
 \leq (|y|+C_B)|\varphi'(y)|
      +\tfrac12m_2\|\varphi''\|_\infty.
 \label{eq:generator-bound}
\end{equation}
For the identity function $\iota(y)=y$ and its square, the values are
\begin{align}
 \mathcal G_{B,F}\iota(y)&=a_{B,F}(y),\label{eq:linear-test}\\
 \mathcal G_{B,F}(\iota^2)(y)&=2ya_{B,F}(y)
          +\int v^2k_{B,F}(y,v)\nu_0(\dd v),\label{eq:quadratic-test}\\
 |\mathcal G_{B,F}(\iota^2)(y)|&\leq3y^2+C_B^2+m_2.
 \label{eq:quadratic-bound}
\end{align}
\end{lemma}

\begin{proof}
We begin with the two deterministic kernels.
Multiplying by $(u-1)^2$, we obtain
\[
 (u-1)^2K(u)=\log u-\frac{2(u-1)}{u+1}.
\]
The derivative of the right-hand side is $(u-1)^2/[u(u+1)^2]$.
The right-hand side itself vanishes at $u=1$, is negative below one and positive above one.
The apparent singularity of $K$ is removable, with value zero.
An antiderivative is
\[
 H(u)=\log(1+u^{-1})-\frac{\log u}{u-1},\qquad
 H(0+)=H(+\infty)=0,\quad H(1)=\log2-1.
\]
It follows that \eqref{eq:K-integrals} holds and that $\int|K|=2(1-\log2)$.
The density of $\nu_0$ is even, and on $v>0$ it equals $\sum_{n\geq1}ne^{-nv}$.
Tonelli's theorem gives \eqref{eq:m2}, since $\int_0^\infty v^2e^{-nv}\dd v=2/n^3$.
Only the finiteness of this series is needed in the proof.

The uniform drift bound comes from the atom added at the distinguished rate $b$.
Indeed, the posterior representation gives
\[
 Z/2\leq bM_{P^{(b)}}(b)\leq1.
\]
Differentiating the beta integral gives
\[
 0\leq\E^{(b)}[-\log\{bM_P(b)\}]
 \leq\log2+\E[-\log Z]
 =\log2+\psi_0(B+1)-\psi_0(1).
\]
The differentiation is legitimate because $|\log z|(1-z)^{B-1}$ is integrable on $(0,1)$.
Since $0\leq\xi_P\leq1$, its integral against $K$ lies in $[-(1-\log2),1-\log2]$.
Substituting these bounds into \eqref{eq:log-drift}, we obtain \eqref{eq:drift-bounds}.
For example, take
\[
 C_B=\max\{|\psi_0(1)|,
                 |\psi_0(B+1)+2-\log2|\}.
\]
Lemma~\ref{lem:parameterized-dirichlet}, followed by the rate pushforward, realizes all posterior probabilities on one fixed probability space as a jointly Borel function of $(B,F,y)$ and the auxiliary randomness.
Together with Lemma~\ref{lem:phase}, this makes the phase integrands jointly Borel, also in $v$ or $u$ as appropriate.
Integration over that fixed space preserves measurability.
The logarithmic expectation is finite by the displayed beta bound, and the phase correction is absolutely bounded by $\|K\|_1$; their positive and negative parts may therefore be integrated separately.
Thus $a_{B,F}(y)$ and $k_{B,F}(y,v)$ are jointly Borel in all their displayed parameters, as well as finite.

For the operator itself, Taylor's formula bounds the jump remainder in absolute value by $\frac12\|\varphi''\|_\infty v^2$, proving \eqref{eq:generator-bound} and absolute convergence.
Linear and quadratic remainders are respectively zero and $v^2$, proving the remaining assertions by $2|y|(|y|+C_B)\leq3y^2+C_B^2$.
\end{proof}

\subsection{Action on the resolvent tests}

The choice of the correction term becomes transparent on the resolvents $\varphi_s(y)=(s+e^y)^{-1}$.
After cancellation, the phase integral is the difference of two logarithms of $M_P$.
Posterior averaging then gives the power derivative in Lemma~\ref{lem:tangent}.

\begin{lemma}[The resolvent identity]
\label{lem:generator}
Assume in addition that $U$ satisfies \eqref{eq:finite-admissible}.
Let $g$ and $h_U$ be the functions associated with $U$ in Lemma~\ref{lem:tangent}.
For $s>0$ put $\varphi_s(y)=(s+e^y)^{-1}$.
Then
\begin{equation}
 F(\mathcal G_{B,F}\varphi_s)
 =\frac{h_U(s)+g(s)}B.
 \label{eq:generator-tangent}
\end{equation}
The generator is integrable against $F$ for $\varphi\in C_c^2(\R)$ and for $\varphi_s$, without assuming a log-rate moment.
It annihilates constants and satisfies the positive minimum property: if a test in its domain attains a global minimum at $y$, then $\mathcal G_{B,F}\varphi(y)\ge0$.
\end{lemma}

\begin{proof}
For compactly supported tests, $(1+|y|)|\varphi'(y)|$ is bounded.
For $\varphi_s$ it is also bounded, since $\varphi_s'(y)=-e^y/(s+e^y)^2$ decays exponentially at both ends; its second derivative is bounded.
Integrability follows from \eqref{eq:generator-bound}.
Constants have zero generator.
At a global minimum the derivative vanishes and the remaining jump integral is nonnegative.

We prove the identity first for a fixed posterior sample.
Fix $b=e^y$ and a deterministic probability $P$ on positive rates.
Define the sample drift
\[
 \widehat a_P(y)
 =-\psi_0(B+1)-1-\log M_P(b)
       +\int_0^\infty\xi_P(bu)K(u)\,\dd u,
\]
whose posterior expectation is $a_{B,F}(y)$.
In the jump term set $u=e^v$, so $\nu_0(\dd v)=\dd u/(u-1)^2$.
Let $\widehat{\mathcal G}_P$ denote the operator in \eqref{eq:generator} with drift $\widehat a_P$ and jump acceptance $\xi_P(be^v)$.
The elementary identity
\begin{multline}
 \frac{(s+bu)^{-1}-(s+b)^{-1}
                  +b\log u/(s+b)^2}{(u-1)^2}
       -\frac{b}{(s+b)^2}K(u)\\
 =\frac{b}{(s+b)^2}
          \left\{\frac{b}{s+bu}-\frac1{1+u}\right\}
 \label{eq:direct-resolvent-cancellation}
\end{multline}
holds for $u>0$, $u\ne1$, and extends by continuity at one.
The first term is absolutely integrable: it is the log-coordinate Taylor remainder just estimated.
The $K$ term is absolutely integrable.
The bracketed kernel on the right has absolute integral $|\log(s/b)|$; the full right-hand side has absolute integral $b(s+b)^{-2}|\log(s/b)|$.
All three integrals are therefore absolutely convergent, so the cancellation is legitimate.

Subtracting the phase representations at $s$ and $b$ and changing variables $t=bu$ gives
\[
 \log M_P(s)-\log M_P(b)
 =\int_0^\infty\xi_P(bu)
       \left\{\frac{b}{s+bu}-\frac1{1+u}\right\}\dd u.
\]
Since $\varphi_s'(y)=-b/(s+b)^2$, the sample drift and jump terms therefore sum to
\begin{equation}
 \widehat{\mathcal G}_P\varphi_s(y)
 =\frac{b}{(s+b)^2}
            \{\psi_0(B+1)+1+\log M_P(s)\}.
 \label{eq:sample-resolvent}
\end{equation}

We next take posterior expectations and integrate in $F(\dd y)$.
The application of the Palm identity to the signed right-hand side is justified first by Tonelli for its absolute value: the resulting expectation is
\[
 \E_{\DP(U)}\left[
 W_P(s)|\psi_0(B+1)+1+\log M_P(s)|\right]<\infty,
\]
because $W_P\leq M_P\leq1/s$ and $M_P(1+|\log M_P|)$ is uniformly bounded.
To interchange the integrals on the left, we also need an absolute bound before averaging the sample drift.
The posterior atom estimate gives
\begin{equation}
 \E^{(b)}|\widehat a_P(y)-y|
 \leq |\psi_0(B+1)+1|+\log2+\psi_0(B+1)-\psi_0(1)+\|K\|_1
 =:C'_B<\infty.
 \label{eq:absolute-posterior-drift}
\end{equation}
Hence the expected absolute sample drift term is bounded by $(|y|+C'_B)|\varphi_s'(y)|$, a bounded function of $y$.
The absolute sample jump integral is at most $m_2\|\varphi_s''\|_\infty/2$.
These estimates justify both the posterior expectations and the signed Palm interchange without a log-rate moment hypothesis.
We obtain
\[
 F(\mathcal G_{B,F}\varphi_s)
 =\E_{\DP(U)}\left[
 W_P(s)\{\psi_0(B+1)+1+\log M_P(s)\}\right].
\]
Equation~\eqref{eq:normalized-tangent} now gives \eqref{eq:generator-tangent}.
The parameter $B$ in the digamma term has not changed under the posterior disintegration.
\end{proof}

\begin{remark}
If $F\in\Ptwo$, then $U=B\exp_*F$ is Thorin-admissible, since
\begin{equation}
 \int\log(1+b^{-1})\,U(\dd b)
 =B F\bigl(\log(1+e^{-y})\bigr)
 \leq B\{\log2+F(|y|)\}<\infty.
 \label{eq:log-moment-admissibility}
\end{equation}
Every finite gamma convolution gives such an $F$.
We shall preserve this second logarithmic moment on finite time intervals; no moment of the positive rate itself is required.
\end{remark}

\section{Evolution of the Thorin measure}
\label{sec:evolution}

We construct the curve of log-rate probabilities from the operator of Section~\ref{sec:tangent}.
The proof has two ingredients: continuity of the averaged operator under weak convergence, and an Euler approximation in which each step is a probability kernel.

For \(B>0\) and \(F\in\mathcal P(\R)\), retain the coefficients of Section~\ref{sec:tangent} and the relation \(U=B\exp_*F\).
Recall that
\begin{equation}
 \mathcal G_{B,F}\varphi(y)
 =a_{B,F}(y)\varphi'(y)
  +\int_{\R}\bigl[
       \varphi(y+v)-\varphi(y)-v\varphi'(y)
    \bigr]k_{B,F}(y,v)\,\nu_0(\dd v),
 \qquad
 \nu_0(\dd v)=\frac{\dd v}{4\sinh^2(v/2)}.
 \label{eq:evolution-generator}
\end{equation}
The compensation in this expression is over all jump sizes.
Lemma~\ref{lem:bounds} gives, for every compact interval \(I\subset(0,\infty)\), a constant \(C_I<\infty\) such that
\begin{equation}
 0\le k_{B,F}(y,v)\le1,\qquad
 |a_{B,F}(y)-y|\le C_I,\qquad
 m_2:=\int_{\R}v^2\,\nu_0(\dd v)<\infty,
 \qquad B\in I.
 \label{eq:evolution-uniform-bounds}
\end{equation}
These bounds hold for every probability \(F\), without a moment assumption.
The continuity result below will permit passage to the limit in the dependence of the coefficients on \(F\).

\subsection{Weak continuity of the averaged generator}

\begin{lemma}[Joint continuity of the averaged generator]\label{lem:continuity}
Let \(B_n,B>0\), \(F_n,F\in\mathcal P(\R)\), and \(y_n,y\in\R\).
Suppose \(B_n\to B\), \(F_n\Rightarrow F\), and \(y_n\to y\).
Then
\[
 a_{B_n,F_n}(y_n)\longrightarrow a_{B,F}(y).
 \]
For every \(\varphi\in C_c^2(\R)\),
\begin{equation}
 \mathcal G_{B_n,F_n}\varphi(y_n)
   \longrightarrow\mathcal G_{B,F}\varphi(y).
 \label{eq:generator-joint-continuity}
\end{equation}
Moreover, \(\mathcal G_{B,F}\varphi(y)\) is bounded uniformly over all probabilities \(F\), all \(y\in\R\), and \(B\) in a fixed compact positive interval.
Consequently the scalar functional
\begin{equation}
 H_\varphi(B,F):=F(\mathcal G_{B,F}\varphi)
 \label{eq:averaged-generator-functional}
\end{equation}
is continuous for the usual topology on \(B>0\) and the narrow topology on probabilities \(F\).
It is uniformly bounded when \(B\) is restricted to a compact positive interval.
\end{lemma}

\begin{proof}
We first couple the posterior probabilities so that their weak convergence holds almost surely.
Use the fixed probability space and the quantile maps of Lemma~\ref{lem:parameterized-dirichlet}.
For a probability \(F\) on \(\R\), put \(Z_j=q_F(S_j)\) and \(V_j=1-T_j^{1/B}\).
Thus the locations are i.i.d. with law \(F\), the break fractions are i.i.d. with law \(\operatorname{Beta}(1,B)\), and the two sequences are independent.
On the common full-probability event in that lemma, the resulting stick-breaking probability has the form
\begin{equation}
 W_j=V_j\prod_{i<j}(1-V_i),
 \qquad
 Q=\sum_{j\ge1}W_j\delta_{Z_j}.
 \label{eq:evolution-stick-breaking}
\end{equation}
The remaining mass \(R_m=\prod_{j\le m}(1-V_j)\) decreases to zero on the common full-probability event specified in that lemma, and \(Q\) has law \(\DP(BF)\), including for atomic base probabilities.

For the sequence \((B_n,F_n)\), set \(Z_{n,j}=q_{F_n}(S_j)\).
Then \(Z_{n,j}\to Z_j\) almost surely for every \(j\).
For completeness, convergence of the inverse distribution functions holds at every continuity point of the limiting inverse: it follows by bracketing that inverse between continuity points of the limiting distribution function and using \(F_n\Rightarrow F\).
The exceptional set has Lebesgue measure zero because an inverse distribution function is monotone.
Also use common uniforms \(T_j\) to set
\[
 V_{n,j}=1-T_j^{1/B_n},
 \qquad V_j=1-T_j^{1/B}.
 \]
Every fixed finite set of weights and locations then converges almost surely.
On that event, for a bounded continuous function \(f\), the difference between the two stick-breaking integrals is bounded by the difference of their first \(m\) terms plus \(\|f\|_\infty(R_{n,m}+R_m)\).
First letting \(n\to\infty\), and then \(m\to\infty\), proves
\[
 Q_n\Rightarrow Q\quad\hbox{almost surely}.
 \]
The event just used concerns only the countably many weights and locations, so it gives weak convergence of the random measures, not merely convergence for one chosen test function.

Use the uniform variable \(T_0\), independent of both coordinate sequences \((S_j)_j\) and \((T_j)_j\), and put
\[
 Z_n^*=1-T_0^{1/B_n},\qquad Z^*=1-T_0^{1/B}.
 \]
By the posterior coupling \eqref{eq:palm-coupling}, the log-rate probabilities
\[
 \widehat P_n=(1-Z_n^*)Q_n+Z_n^*\delta_{y_n},
 \qquad
 \widehat P=(1-Z^*)Q+Z^*\delta_y
 \]
have laws \(\DP(B_nF_n+\delta_{y_n})\) and \(\DP(BF+\delta_y)\), respectively.
They converge weakly almost surely.
Their rate-coordinate pushforwards \(P_n=\exp_*\widehat P_n\) and \(P=\exp_*\widehat P\) are the posterior probabilities appearing in the coefficients.

The next point is convergence of the phases against integrable kernels.
For each coupled sample, set \(b_n=e^{y_n}\), \(b=e^y\), and
\begin{equation}
 \begin{split}
 R_n(s)&=b_nM_{P_n}(b_ns)
       =\int_{\R}\frac{\widehat P_n(\dd z)}
                           {s+e^{z-y_n}},\\
 \eta_n(u)&=\xi_{P_n}(b_nu),\qquad
 R(s)=\int_{\R}\frac{\widehat P(\dd z)}
                         {s+e^{z-y}},\qquad
 \eta(u)=\xi_P(bu).
 \end{split}
 \label{eq:scaled-posterior-transform}
\end{equation}
For fixed \(s>0\), the integrand defining \(R_n(s)\) is bounded and continuous in \(z\).
Its derivative with respect to \(y_n\) has absolute value at most \(1/(4s)\), uniformly in \(z\).
Hence \(R_n(s)\to R(s)>0\).
The anchored phase representation \eqref{eq:phase-anchor}, after scaling the rate variable, gives
\begin{equation}
 \log R_n(s)-\log R_n(1)
   =\int_0^\infty\eta_n(u)
       \left(\frac1{s+u}-\frac1{1+u}\right)\dd u,
 \qquad 0\le\eta_n\le1.
 \label{eq:scaled-phase-anchor}
\end{equation}

These identities imply \(\eta_n\to\eta\) weak-star in \(L^\infty(0,\infty)\); explicitly,
\[
 \int_0^\infty f(u)\eta_n(u)\,\dd u
 \longrightarrow\int_0^\infty f(u)\eta(u)\,\dd u
 \qquad\text{for every }f\in L^1(0,\infty).
\]
We prove this assertion by compactness and uniqueness.
Since \(L^1(0,\infty)\) is separable, any subsequence has a further subsequence whose integrals converge on a countable dense set of \(L^1\) tests.
The uniform bound one extends the limits to a bounded linear functional on \(L^1\), represented by a function \(\eta_*\) with \(0\le\eta_*\le1\).
The kernels in \eqref{eq:scaled-phase-anchor} belong to \(L^1\).
Passing to the limit shows that \(\eta_*\) represents \(\log R(s)-\log R(1)\).
Uniqueness of the exponential Stieltjes phase representation, equivalently that for the reciprocal complete Bernstein function, gives \(\eta_*=\eta\) almost everywhere; see \cite[Theorems~6.10 and~7.3]{SSV2010}.
Every subsequential limit is therefore the same, which proves convergence against every \(L^1\) kernel.
This argument uses only integrals of the phases, so their exceptional boundary values, including after rate scaling, play no role.

We can now pass to the limit in the posterior averages.
The drift formula \eqref{eq:log-drift} takes the form
\begin{equation}
 a_{B_n,F_n}(y_n)
 =y_n-\psi_0(B_n+1)-1
   +\E\left[-\log R_n(1)
           +\int_0^\infty\eta_n(u)K(u)\,\dd u\right].
 \label{eq:continuity-drift-formula}
\end{equation}
The kernel \(K\) belongs to \(L^1(0,\infty)\) by Lemma~\ref{lem:bounds}.
The phase integrals converge almost surely by weak-star convergence and are bounded in absolute value by \(\|K\|_1\).
Moreover,
\[
 \frac{Z_n^*}{2}\le R_n(1)\le1.
 \]
Choose \(B_*<\infty\) with \(B_n\le B_*\) for all \(n\).
The common-uniform coupling gives the integrable bound
\[
 0\le-\log R_n(1)
 \le\log2-\log(1-T_0^{1/B_*}).
 \]
Integrability follows from the logarithmic moment of a \(\operatorname{Beta}(1,B_*)\) variable.
Dominated convergence in \eqref{eq:continuity-drift-formula} proves continuity of the drift.

For the jump part, put \(u=e^v\).
The integral against a single posterior phase has kernel
\begin{equation}
 J_\varphi(y,u)
 =\frac{\varphi(y+\log u)-\varphi(y)
              -\log u\,\varphi'(y)}{(u-1)^2}.
 \label{eq:continuity-jump-kernel}
\end{equation}
Taylor's theorem gives
\[
 |J_\varphi(y,u)|
 \le\frac{\|\varphi''\|_\infty}{2}
                 \frac{(\log u)^2}{(u-1)^2}.
 \]
The dominating function is integrable on \((0,\infty)\): at \(u=1\) the quotient has a removable finite limit, and at zero and infinity it is bounded by integrable multiples of \((\log u)^2\) and \((\log u)^2/u^2\), respectively.
Consequently \(J_\varphi(y_n,\cdot)\to J_\varphi(y,\cdot)\) in \(L^1\).
This convergence and the weak-star convergence of \(\eta_n\) imply convergence of the random jump integrals.
Their uniform bound permits taking expectations.
This proves \eqref{eq:generator-joint-continuity}.

Finally, we integrate against the varying base probability.
For \(\varphi\in C_c^2(\R)\),
\begin{equation}
 |\mathcal G_{B,F}\varphi(y)|
 \le (|y|+C_I)|\varphi'(y)|
       +\frac{m_2}{2}\|\varphi''\|_\infty.
 \label{eq:compact-test-generator-bound}
\end{equation}
Its right side is bounded uniformly in \(y\).
Joint continuity implies that \(\mathcal G_{B_n,F_n}\varphi\to\mathcal G_{B,F}\varphi\) uniformly on compact \(y\)-sets: otherwise a sequence of points violating uniform convergence has a convergent subsequence contradicting \eqref{eq:generator-joint-continuity}.
Tightness of \(F_n\), this local uniform convergence, the uniform global bound, and weak convergence against the bounded continuous limiting function give \(H_\varphi(B_n,F_n)\to H_\varphi(B,F)\).
\end{proof}

\subsection{Existence by positive Euler approximation}

\begin{theorem}[Existence of a log-rate evolution]\label{thm:evolution}
Let \(B_0>0\), \(F_0\in\Ptwo\), and \(T>0\).
Set \(B_t=B_0e^{-t}\).
There is a narrowly continuous curve \((F_t)_{0\le t\le T}\) of probabilities in \(\Ptwo\) such that
\begin{equation}
 \sup_{0\le t\le T}F_t(y^2)<\infty
 \label{eq:evolution-second-moment}
\end{equation}
and
\begin{equation}
 F_t(\varphi)-F_0(\varphi)
   =\int_0^t F_r(\mathcal G_{B_r,F_r}\varphi)\,\dd r,
 \qquad
 0\le t\le T,\quad \varphi\in C_c^2(\R).
 \label{eq:evolution-weak-equation}
\end{equation}
\end{theorem}

The displayed identity is the weak evolution equation used in this paper.
Its construction gives the candidate Thorin measures; Section~\ref{sec:identification} will identify their value distributions.

\begin{proof}
All constants below may depend on \(B_0,T\) and \(F_0(y^2)\), but not on the mesh or on the intermediate probability measures.
The mass parameter lies in the fixed compact interval \(I=[B_0e^{-T},B_0]\).

\emph{The discrete kernel.}
Take \(h=T/N\) and \(\epsilon=\sqrt h\), with \(h\) sufficiently small.
Freeze \(B,F\) during one step, and abbreviate
\begin{equation}
 \begin{split}
 \lambda_\epsilon(y)
   &=\int_{|v|>\epsilon}k_{B,F}(y,v)\,\nu_0(\dd v),\\
 m_\epsilon(y)
   &=\int_{|v|>\epsilon}v k_{B,F}(y,v)\,\nu_0(\dd v),\\
 a_\epsilon(y)&=a_{B,F}(y)-m_\epsilon(y),
 \qquad
 p_\epsilon(y)=1-h\lambda_\epsilon(y).
 \end{split}
 \label{eq:euler-truncated-coefficients}
\end{equation}
The universal dominating measure gives, for \(0<\epsilon\le1\),
\begin{equation}
 \begin{split}
 \lambda_\epsilon(y)
 &\le\int_{|v|>\epsilon}\nu_0(\dd v)
   =\frac{2}{e^\epsilon-1}\le\frac2\epsilon,\\
 |m_\epsilon(y)|
 &\le\int_{|v|>\epsilon}|v|\,\nu_0(\dd v)\\
 &=2\left\{\frac{\epsilon}{e^\epsilon-1}
                      -\log(1-e^{-\epsilon})\right\}
   \le C(1+|\log\epsilon|).
 \end{split}
 \label{eq:euler-truncation-bounds}
\end{equation}
The equalities follow by integrating on the positive half-line and using the evenness of \(\nu_0\); on that half-line its density is \(-\frac{\dd}{\dd v}(e^v-1)^{-1}\).
In particular,
\begin{equation}
 |a_\epsilon(y)|
 \le |y|+C_I+C(1+|\log\epsilon|).
 \label{eq:euler-drift-bound}
\end{equation}
If \(h\le1/16\), then \(p_\epsilon(y)\ge1-2\sqrt h\ge1/2\), uniformly in \(B,F,y\).
Define
\begin{equation}
 \begin{split}
 \Pi_h^{B,F}(y,\dd z)
 ={}&p_\epsilon(y)\,
      \delta_{\,y+h a_\epsilon(y)/p_\epsilon(y)}(\dd z)\\
 &+h\int_{|v|>\epsilon}
       \delta_{y+v}(\dd z)\,k_{B,F}(y,v)\,\nu_0(\dd v).
 \end{split}
 \label{eq:positive-euler-kernel}
\end{equation}
Measurability of the coefficients makes this a Borel kernel.
Both terms are nonnegative and their total mass is exactly one.
Thus $\Pi_h^{B,F}(y,\cdot)$ is a probability for every $y$.
For a test $\varphi$ and a probability $F$, our kernel notation is
\[
 \Pi_h^{B,F}\varphi(y)=\int\varphi(z)\Pi_h^{B,F}(y,\dd z),
 \qquad
 (F\Pi_h^{B,F})(A)=\int\Pi_h^{B,F}(y,A)F(\dd y).
\]
The superscript \(h\) labels the time step, and \(F_j^h\) denotes the approximate law at time \(jh\).
Starting from \(F_0^h=F_0\), we can therefore define recursively
\begin{equation}
 F_{j+1}^h
   =F_j^h\Pi_h^{B_{jh},F_j^h},
 \qquad j=0,\ldots,N-1.
 \label{eq:positive-euler-recursion}
\end{equation}
Since the coefficients are defined for every probability, each successive step is well defined.

For one step from \(y\), write \(\Delta=z-y\).
Direct calculation from \eqref{eq:positive-euler-kernel} gives
\begin{equation}
 \begin{split}
 \E(\Delta\mid y)&=h a_{B,F}(y),\\
 \E(\Delta^2\mid y)
   &=h\int_{|v|>\epsilon}v^2 k_{B,F}(y,v)\,\nu_0(\dd v)
       +\frac{h^2a_\epsilon(y)^2}{p_\epsilon(y)}.
 \end{split}
 \label{eq:euler-exact-moments}
\end{equation}
The cancellation in the first identity is essential: it controls the mean increment although the untruncated first absolute jump moment is infinite.

\emph{Moment control.}
Put \(M_j^h=F_j^h(y^2)\).
From \eqref{eq:evolution-uniform-bounds},
\[
 2y a_{B,F}(y)\le3y^2+C_I^2.
 \]
Using \(p_\epsilon\ge1/2\), \eqref{eq:euler-drift-bound}, and \eqref{eq:euler-exact-moments}, we obtain
\[
 M_{j+1}^h
 \le (1+3h+C h^2)M_j^h
       +C h+C h^2(1+|\log h|^2).
 \]
Because \(h(1+|\log h|^2)\) is bounded for sufficiently small \(h\), this implies
\begin{equation}
 M_{j+1}^h\le(1+C_T h)M_j^h+C_T h,
 \qquad
 \sup_{\substack{h\ \mathrm{small}\\0\le j\le N}}M_j^h
       \le C_{T,B_0,F_0}<\infty.
 \label{eq:euler-uniform-moments}
\end{equation}
The second inequality follows by iterating the scalar recursion, or by summing its geometric series.
This also proves by induction that each discrete law has a finite second moment, so every moment calculation just used is justified.

\emph{Compactness.}
We realize a finite Markov chain with initial law \(F_0\) and the kernels in \eqref{eq:positive-euler-recursion}.
This finite chain is obtained by iterating the Borel kernels just constructed.
Writing \(Y_j\) for its state at step \(j\), let
\[
 D_{j+1}
  =Y_{j+1}-Y_j-h a_{B_{jh},F_j^h}(Y_j).
 \]
With respect to the filtration $\mathcal F_j=\sigma(Y_0,\ldots,Y_j)$, these are square-integrable martingale differences, by \eqref{eq:euler-exact-moments}.
The same formula and \eqref{eq:euler-uniform-moments} yield
\[
 \E D_{j+1}^2\le C_T h,
 \qquad
 \E\bigl|a_{B_{jh},F_j^h}(Y_j)\bigr|^2\le C_T.
 \]
For \(i<j\), martingale orthogonality gives \(\E|\sum_{\ell=i}^{j-1}D_{\ell+1}|^2\le C_T(j-i)h\).
Cauchy--Schwarz bounds the squared drift sum by \(C_T((j-i)h)^2\).
Consequently
\begin{equation}
 \E|Y_j-Y_i|^2
   \le C_T\bigl\{(j-i)h+((j-i)h)^2\bigr\}.
 \label{eq:euler-time-moment}
\end{equation}

We interpolate the laws rather than the sample paths: for \(t=(j+\theta)h\), \(0\le\theta\le1\), set
\[
 \widetilde F_t^h=(1-\theta)F_j^h+\theta F_{j+1}^h.
 \]
Also let \(\overline F_t^h=F_j^h\) for \(jh\le t<(j+1)h\), with both interpolations equal to \(F_N^h\) at \(T\).
Use the bounded Lipschitz metric
\[
 d_{\mathrm{BL}}(F,G)
 =\sup_{\substack{\|f\|_\infty\le1\\ \operatorname{Lip}(f)\le1}}
                  |F(f)-G(f)|,
 \]
where $\operatorname{Lip}(f)=\sup_{x\ne y}|f(x)-f(y)|/|x-y|$.
This metric metrizes narrow convergence of probabilities on \(\R\).
The coupling in \eqref{eq:euler-time-moment}, followed by Cauchy--Schwarz, gives the uniform estimates
\begin{equation}
 \begin{split}
 d_{\mathrm{BL}}(\widetilde F_t^h,\widetilde F_s^h)
      &\le C_T\sqrt{|t-s|+h},\\
 \sup_{0\le t\le T}
 d_{\mathrm{BL}}(\widetilde F_t^h,\overline F_t^h)
      &\le C_T\sqrt h .
 \end{split}
 \label{eq:euler-time-modulus}
\end{equation}
For the first estimate, compare each interpolated law to a neighboring mesh law and apply \eqref{eq:euler-time-moment}; the squared time increment is absorbed into the constant depending on \(T\).
The second is the adjacent-step case.

The interpolated laws also satisfy \eqref{eq:euler-uniform-moments}. Thus \(\widetilde F_t^h(|y|>R)\le C_{T,B_0,F_0}/R^2\).
They are therefore uniformly tight.
Take a sequence of meshes tending to zero.
At each time in a countable dense subset of \([0,T]\) containing its endpoints, tightness gives a weakly convergent subsequence; diagonal selection gives one subsequence working at all these times.
On the real line this selection can equivalently be obtained from monotonicity of distribution functions and the displayed uniform tail bound.

The first estimate in \eqref{eq:euler-time-modulus} is an asymptotic, uniform time modulus.
Its limit on the dense time set is at most \(C_T\sqrt{|t-s|}\).
Uniform tightness and this bound extend the limiting probabilities uniquely to every time, giving a narrowly continuous curve \(F_t\).
Comparing any time to a finite dense time grid, and then using \eqref{eq:euler-time-modulus}, shows that along the chosen subsequence
\begin{equation}
 \sup_{0\le t\le T}
 d_{\mathrm{BL}}(\widetilde F_t^h,F_t)\longrightarrow0,
 \qquad
 \sup_{0\le t\le T}
 d_{\mathrm{BL}}(\overline F_t^h,F_t)\longrightarrow0.
 \label{eq:euler-uniform-narrow-limit}
\end{equation}
The limits remain probability measures: the uniform tail bound precludes loss of mass.
Approximating a bounded continuous function uniformly on a large compact interval by a Lipschitz function, and controlling the complementary tails, also gives uniform-in-time convergence for every fixed bounded continuous test.
Finally, apply this convergence to \(\min\{y^2,R\}\), and then let \(R\to\infty\).
It follows that \(F_t\in\Ptwo\) for every \(t\) and that \eqref{eq:evolution-second-moment} holds.
The curve is thus narrowly continuous and has the required uniform second-moment bound.

\emph{Consistency.}
Fix \(\varphi\in C_c^2(\R)\).
Taylor expansion at the no-jump location in \eqref{eq:positive-euler-kernel} gives a remainder bounded by
\[
 \frac{h^2\|\varphi''\|_\infty
                      a_\epsilon(y)^2}{2p_\epsilon(y)}.
 \]
The first-order term is \(h a_\epsilon(y)\varphi'(y)\).
Adding the retained jump terms converts it exactly into the truncated, fully compensated operator.
For the missing small jumps,
\[
 \int_{|v|\le\epsilon}v^2\,\nu_0(\dd v)\le2\epsilon,
 \]
because \(4\sinh^2(v/2)\ge v^2\).
We therefore have the pointwise consistency estimate
\begin{equation}
 \left|
 \frac{\Pi_h^{B,F}\varphi(y)-\varphi(y)}h
       -\mathcal G_{B,F}\varphi(y)
 \right|
 \le C_I\|\varphi''\|_\infty
       \left[\epsilon+
              h\{y^2+1+|\log\epsilon|^2\}\right].
 \label{eq:euler-consistency}
\end{equation}
Here and below the value of the constant \(C_I\) may be enlarged; it remains independent of \(F,y,h\).
Integrating against \(F_j^h\), using \eqref{eq:euler-uniform-moments}, and summing the steps yields an error, uniform in terminal time, at most
\begin{equation}
 C_{T,\varphi}\bigl[\sqrt h+h(1+|\log h|^2)\bigr]
   \longrightarrow0.
 \label{eq:euler-total-error}
\end{equation}
More explicitly, let \(\overline B_r^h=B_{jh}\) on the \(j\)-th mesh interval.
Telescoping the discrete equations and, on the final partial interval, taking the same convex combination as in \(\widetilde F_t^h\), gives
\begin{equation}
 \widetilde F_t^h(\varphi)-F_0(\varphi)
  =\int_0^t
       H_\varphi(\overline B_r^h,\overline F_r^h)\,\dd r
       +e_h(t),
 \qquad
 \sup_{0\le t\le T}|e_h(t)|\longrightarrow0 .
 \label{eq:euler-integrated-consistency}
\end{equation}
The convex interpolation therefore gives the same vanishing error at every terminal time.

\emph{Passage to the limit.}
Lemma~\ref{lem:continuity} says that \(H_\varphi\) is continuous and uniformly bounded on \(I\times\mathcal P(\R)\).
Equations \eqref{eq:euler-uniform-narrow-limit} and \(\sup_r|\overline B_r^h-B_r|\to0\) imply
\begin{equation}
 \sup_{0\le r\le T}
 \left|
 H_\varphi(\overline B_r^h,\overline F_r^h)
       -H_\varphi(B_r,F_r)
 \right|\longrightarrow0.
 \label{eq:euler-nonlinear-limit}
\end{equation}
For the uniform assertion, suppose otherwise and choose times \(r_h\) at which the difference is bounded away from zero.
Along a further subsequence, \(r_h\to r\).
Uniform narrow convergence and continuity of \(F_\cdot\) give \(\overline F_{r_h}^h\Rightarrow F_r\), while \(\overline B_{r_h}^h\to B_r\).
Continuity of \(H_\varphi\) gives a contradiction, also using continuity of \(r\mapsto H_\varphi(B_r,F_r)\).

Passing to the limit in \eqref{eq:euler-integrated-consistency}, we obtain \eqref{eq:evolution-weak-equation} for the fixed test \(\varphi\), at every time.
The subsequence used to construct \(F_\cdot\) was selected only through weak compactness, independently of \(\varphi\).
The same argument therefore applies to every \(\varphi\in C_c^2(\R)\), establishing the stated equation on that entire test class.
\end{proof}

\begin{remark}
The resulting rate measures are Thorin-admissible throughout the finite time interval.
Indeed,
\[
 \int_0^\infty\log(1+1/b)\,
                B_t\exp_*F_t(\dd b)
 =B_tF_t\bigl(\log(1+e^{-y})\bigr)
 \le B_t\bigl(\log2+F_t(|y|)\bigr)<\infty.
 \]
Thus the moment estimate gives admissibility at every time.
We turn next to the distributions represented by these Thorin measures.
\end{remark}

\section{Identification by transport}
\label{sec:identification}

The evolution constructed above gives a family of GGC distributions.
We show that its weak equation determines their Laplace transforms in the same way as deterministic powering.
A logarithmic change of the value variable then identifies the entire family.

\begin{theorem}[Identification of the weak evolution]
\label{thm:identification}
Let $B_0>0$ and $0\le T<\infty$.
Let $(F_t)_{0\le t\le T}$ be a narrowly continuous curve of probability measures on $\R$ satisfying
\begin{equation}
 K_T=\sup_{0\le t\le T}\int_\R y^2F_t(\dd y)<\infty.
 \label{eq:id-assumptions}
\end{equation}
Set $B_t=B_0e^{-t}$ and $U_t=B_t\exp_*F_t$.
Assume that the operator $\mathcal G$ defined in \eqref{eq:generator} satisfies, for every $t\in[0,T]$ and every $\varphi\in C_c^2(\R)$,
\begin{equation}
 F_t(\varphi)-F_0(\varphi)
 =\int_0^t F_u(\mathcal G_{B_u,F_u}\varphi)\,\dd u.
 \label{eq:id-weak-rate}
\end{equation}
Then each $U_t$ is an admissible finite Thorin measure.
Let $\mu_t$ be its zero-drift GGC law.
If $X_0$ has law $\mu_0$, then
\begin{equation}
 \mu_t=\Law(X_0^{e^t}),\qquad 0\le t\le T.
 \label{eq:id-conclusion}
\end{equation}
\end{theorem}

The hypotheses are exactly those supplied by Theorem~\ref{thm:evolution}.
The main analytic point is integration of the resolvent evolution down to the zero Laplace argument.
We establish a uniform logarithmic moment for this purpose, and then obtain the transport equation by approximation of its test functions.

\begin{proof}
For $T=0$, admissibility follows from \eqref{eq:log-moment-admissibility} and the distributional identity is immediate.
We may therefore assume $T>0$.
We distinguish the log-rate probabilities $F_t$ from the value distributions $\mu_t$ throughout the argument.
Lemma~\ref{lem:bounds}, on the compact mass interval $I=[B_0e^{-T},B_0]$, gives a constant $C_T<\infty$ such that
\begin{equation}
 |a_{B_t,F_t}(y)-y|\le C_T,\qquad 0\le k_{B_t,F_t}(y,v)\le1,
 \qquad m_2:=\int_\R v^2\nu_0(\dd v)<\infty.
 \label{eq:id-generator-bounds}
\end{equation}
The compensation in $\mathcal G$ is over the whole jump line.

\medskip\noindent\emph{Admissibility and transform continuity.}
For $s>0$, define
\begin{equation}
 \begin{aligned}
 \ell_s(y)&=\log(1+se^{-y}),&
 \Psi_t(s)&=B_tF_t(\ell_s),\\
 L_t(s)&=\exp\{-\Psi_t(s)\},&
 g_t(s)&=B_t\int_\R\frac{F_t(\dd y)}{s+e^y}.
 \end{aligned}
 \label{eq:id-transforms}
\end{equation}
The elementary bound
\begin{equation}
 0\le\ell_s(y)\le\log(1+s)+y_-
 \label{eq:id-log-test-bound}
\end{equation}
shows that $\int\log(1+1/b)\,U_t(\dd b)<\infty$.
Consequently $U_t$ defines the asserted GGC law $\mu_t$, with transform $L_t$, and differentiation at positive $s$ gives $g_t=\partial_s\Psi_t=-\partial_s\log L_t$.
In particular, no first moment of the positive rate $e^y$ is needed.

The moment assumption implies, uniformly in $t$,
\begin{equation}
 F_t\{|y|>R\}\le\frac{K_T}{R^2},\qquad
 \int_{|y|>R}|y|F_t(\dd y)\le\frac{K_T}{R},\qquad R>0.
 \label{eq:id-rate-tails}
\end{equation}
Truncating the continuous function $\ell_s$, whose growth is at most linear, therefore shows that $t\mapsto\Psi_t(s)$ and $t\mapsto L_t(s)$ are continuous.
For each fixed $t$, dominated convergence, with dominating function $\ell_1$ when $s\le1$, gives $\Psi_t(s)\downarrow0$ as $s\downarrow0$.
Thus $L_t(0)=1$.

\medskip\noindent\emph{Logarithmic moments of the value distributions.}
The following estimate controls both the behavior near zero and the upper tail of the value distributions:
\begin{equation}
 \sup_{0\le t\le T}\int_{(0,\infty)}|\log x|\,\mu_t(\dd x)
 \le C_{\log,T}<\infty,
 \label{eq:id-log-moment}
\end{equation}
where the constant depends only on $B_0,T,K_T$.
This is the integrability needed at the zero Laplace endpoint; an $|x\log x|$ moment is not assumed.

For each fixed $t$, let $X_t$ have distribution $\mu_t$; no joint process $(X_t)$ is assumed.
Here $t$ indexes the evolution in time, whereas the earlier notation $X_s$ denotes an exponential tilt with parameter $s$.
For the negative part, Lemma~\ref{lem:gamma-dirichlet}, applied in log-rate coordinates, gives
\[
 X_t\overset{d}=G_tM_t,\qquad
 M_t=\int_\R e^{-y}Q_t(\dd y),\qquad
 Q_t\sim\DP(B_tF_t),
\]
with $G_t\sim\operatorname{Gamma}(B_t,1)$ independent of $Q_t$; only the representation at each fixed $t$ is needed.
The logarithmic integrability already verified makes $M_t$ finite almost surely, and it is strictly positive.
In particular, $\mu_t((0,\infty))=1$.
Moreover, $\E Q_t(|y|)=F_t(|y|)<\infty$.
On this almost-sure integrability event, put $m=\int y\,Q_t(\dd y)$ and integrate $e^{-(y-m)}\ge1-(y-m)$ to obtain $e^mM_t\ge1$.
Hence
\[
 (\log M_t)^-\le\left(\int y\,Q_t(\dd y)\right)^+
 \le\int y_+\,Q_t(\dd y),
 \qquad \E(\log M_t)^-\le\sqrt{K_T}.
\]
For a gamma variable $G\sim\operatorname{Gamma}(B,1)$,
\[
 \E(\log G)^-
 \le\frac1{\Gamma(B)}\int_0^1(-\log x)x^{B-1}\,\dd x
 =\frac1{\Gamma(B)B^2}.
\]
This is bounded uniformly for $B\in I$.
Since $(\log X_t)^-\le(\log G_t)^-+(\log M_t)^-$ under this coupling, the negative part of \eqref{eq:id-log-moment} is controlled.

For the positive part, Tonelli's theorem and the elementary identity
\[
 \log(1+x)=\int_0^\infty\frac{e^{-r}}r(1-e^{-rx})\,\dd r
 \qquad (x\ge0)
\]
give, using $1-e^{-z}\le z$ for $z\ge0$,
\begin{equation}
 \begin{aligned}
 \E\log(1+X_t)
 &=\int_0^\infty\frac{e^{-r}}r(1-L_t(r))\,\dd r
 \le\int_{(0,\infty)}\mathcal{I}(b)\,U_t(\dd b),\\
 \mathcal{I}(b)&:=\int_0^\infty\frac{e^{-r}}r\log(1+r/b)\,\dd r
 \le4\{1+(\log_+(1/b))^2\}.
 \end{aligned}
 \label{eq:id-positive-log}
\end{equation}
Here the logarithmic identity follows, for example, by writing $(1-e^{-rx})/r=\int_0^x e^{-rz}\,\dd z$ and integrating first in $r$.
We next prove the bound on $\mathcal{I}$ for all positive rates.
If $b\ge1$, then $\mathcal{I}(b)\le1/b$, by $\log(1+r/b)\le r/b$.
If $0<b<1$, put $a=\log(1/b)$.
The part $0<r<b$ is at most one.
On $b<r<1$, use $\log(1+r/b)\le\log2+\log(r/b)$; its integral after replacing $e^{-r}$ by one is at most $a\log2+a^2/2$.
On $r>1$, use
\[
 \log(1+r/b)\le\log(1+r)+a\le r+a,\qquad r^{-1}\le1,
\]
which bounds this part by $e^{-1}(1+a)$.
The sum is at most $4(1+a^2)$.
Thus \eqref{eq:id-positive-log} is bounded by $4B_0(1+K_T)$.
Combining this bound with the negative-part estimate, we obtain \eqref{eq:id-log-moment}; one possible choice is
\[
 C_{\log,T}=
 \sup_{B\in I}\frac1{\Gamma(B)B^2}
 +\sqrt{K_T}+4B_0(1+K_T).
\]

The family $(\mu_t)$ is narrowly continuous on $(0,\infty)$.
Indeed, \eqref{eq:id-log-moment} bounds its mass outside $[e^{-R},e^R]$ by $C_{\log,T}/R$.
If $t_n\to t$, tightness and diagonal selection of distribution functions give a weakly convergent subsequence, and this tail bound ensures that its limit is a probability on $(0,\infty)$.
Its transform is $L_t$, by the transform continuity proved above.
These transforms determine a probability uniquely: push the measure forward by $x\mapsto e^{-x}$ to $[0,1]$; the transform values at positive integers give all its positive integer moments, and its zeroth moment is one.
Uniform polynomial approximation of continuous functions then gives uniqueness.
The polynomial approximation argument used below also proves this uniform approximation assertion.
Every subsequential limit is therefore $\mu_t$, proving narrow continuity.

\medskip\noindent\emph{Extension of the log-rate test class.}
The resolvent tests used in the next step are not compactly supported.
We therefore first justify their use in the weak equation.
The weak equation \eqref{eq:id-weak-rate} extends to every $\varphi\in C^2(\R)$ with at most linear growth and with bounded first and second derivatives.
Choose $\chi\in C_c^\infty(\R)$ equal to one on $[-1,1]$ and zero outside $[-2,2]$, and set $\varphi_R(y)=\chi(y/R)\varphi(y)$ for $R\ge1$.
The functions $\varphi_R'$ and $\varphi_R''$ are bounded uniformly in $R$: in the product rule the extra factors $R^{-1}$ and $R^{-2}$ multiply a function of at most linear growth on $|y|\le2R$.
Taylor's formula with full linear compensation therefore gives
\begin{equation}
 \begin{aligned}
 |\mathcal G_{B_t,F_t}\varphi_R(y)|
 &\le (|y|+C_T)|\varphi_R'(y)|
       +\frac{m_2}{2}\|\varphi_R''\|_\infty
 \le C_{T,\varphi}(1+|y|),\\
 \mathcal G_{B_t,F_t}\varphi_R(y)
 &\longrightarrow\mathcal G_{B_t,F_t}\varphi(y).
 \end{aligned}
 \label{eq:id-domain-extension}
\end{equation}
For the second assertion, the compensated jump remainder is bounded by a constant times $v^2$ for every $v$.
The finite second moment in \eqref{eq:id-generator-bounds} therefore controls both the singularity at zero and the tails.
Dominated convergence in $F_t$ and then in time, using $\sup_tF_t(|y|)\le\sqrt{K_T}$, proves the extended weak equation.
The left side converges by the same linear-growth bound.
Also \eqref{eq:id-rate-tails} makes $F_t(\varphi)$ continuous in time.
The extended equation is thus an absolutely continuous scalar identity.

In particular, this applies to
\[
 \varphi_s(y)=(s+e^y)^{-1},\qquad \ell_s(y)=\log(1+se^{-y}),\qquad s>0.
\]
The first function and its first two derivatives are bounded.
For the second, $|\ell_s'|\le1$ and $|\ell_s''|\le1/4$.
Thus the logarithmic test defining $\Psi_t$ is legitimate as well as the resolvent used below.

\medskip\noindent\emph{The resolvent equation.}
Let $X$ denote the coordinate variable on $(0,\infty)$, and write $\E_t H(X)=\int H(x)\mu_t(\dd x)$.
For an integrable test under the exponentially tilted law, set
\[
 \E_{t,s}H(X)=\frac{\int H(x)e^{-sx}\,\mu_t(\dd x)}{L_t(s)}
 \qquad(s>0).
\]
Thus $\E_{t,s}$ is expectation under the normalized exponential tilt of $\mu_t$.
Put
\begin{equation}
 \begin{aligned}
 A_t(s)&=\E_{t,s}[X\log X],\\
 h_t(s)&=\left.\frac{\partial}{\partial q}
 \left[-\partial_s\log\E_t e^{-sX^q}\right]\right|_{q=1}
 =\frac{\dd}{\dd s}\{sA_t(s)\}.
 \end{aligned}
 \label{eq:id-current-tangent}
\end{equation}
Here the power derivative is taken at the current law $\mu_t$, to which Lemma~\ref{lem:tangent} applies.
The last identity can also be obtained by differentiating the normalized exponential tilt.
All derivatives involved are bounded functions of $x$ after multiplication by $e^{-sx^q}$ when $s$ and $q$ range over compact positive intervals.

By Lemma~\ref{lem:generator}, the normalized resolvent identity is
\[
 F_t(\mathcal G_{B_t,F_t}\varphi_s)=\frac{h_t(s)+g_t(s)}{B_t}.
\]
Since $g_t(s)=B_tF_t(\varphi_s)$, the extended weak equation and $B_t'=-B_t$ imply
\begin{equation}
 g_t(s)-g_0(s)=\int_0^t h_u(s)\,\dd u,
 \qquad s>0.
 \label{eq:id-resolvent-evolution}
\end{equation}
The decay of the total Thorin mass thus cancels the extra $g_t(s)$ in the generator identity.

\medskip\noindent\emph{The endpoint at zero.}
To recover the Laplace exponent from its $s$-derivative, we must integrate \eqref{eq:id-resolvent-evolution} down to $s=0$.
Both absolute integrability and the boundary value are needed to recover the normalized Laplace exponent.
Differentiating a normalized tilt in \eqref{eq:id-current-tangent} gives
\begin{equation}
 h_t(s)=A_t(s)-s\E_{t,s}[X^2\log X]
                  +s\E_{t,s}[X]A_t(s).
 \label{eq:id-tangent-expanded}
\end{equation}
We first record joint measurability in time and the Laplace argument.
Fix $0<\sigma<S<\infty$.
If $f_s(x)$ is any one of
\[
 e^{-sx},\quad xe^{-sx},\quad
             x\log x\,e^{-sx},\quad x^2\log x\,e^{-sx},
\]
the map $s\mapsto f_s$ is continuous in the supremum norm on $(0,\infty)$ for $s\in[\sigma,S]$.
Indeed, differentiation in $s$ adds one factor $-x$, and the resulting derivatives are bounded uniformly in $x>0$ and $s\in[\sigma,S]$; the mean value theorem then gives the assertion.
If $(t_n,s_n)\to(t,s)$ in $[0,T]\times[\sigma,S]$, the narrow continuity of the value laws proved above gives
\[
 |\mu_{t_n}(f_{s_n})-\mu_t(f_s)|
 \le\|f_{s_n}-f_s\|_\infty
       +|\mu_{t_n}(f_s)-\mu_t(f_s)|\longrightarrow0.
\]
Since $L_t(s)>0$, division by the jointly continuous denominator $L_t(s)$ preserves continuity.
Equation~\eqref{eq:id-tangent-expanded} therefore shows that $(t,s)\mapsto A_t(s)$ and $(t,s)\mapsto h_t(s)$ are jointly continuous on $[0,T]\times(0,\infty)$, and in particular jointly Borel.
We handle the endpoint $s=0$ by the integral estimates below.

Also, $t\mapsto\mu_t$ is a Borel probability kernel: evaluation on an open set is measurable by bounded continuous approximation, and the monotone-class theorem extends this to every Borel set.
Thus nonnegative Borel integrands may be integrated against $\dd t\,\dd s\,\mu_t(\dd x)$ using Tonelli's theorem.
This supplies the measurability needed for the estimates below, including those with the unbounded test $|\log x|$.

Fix $s_0>0$.
Equations \eqref{eq:id-transforms} and \eqref{eq:id-log-test-bound} give the uniform lower bound
\begin{equation}
 D:=\exp\{-B_0[\log(1+s_0)+\sqrt{K_T}]\}
 \le L_t(s_0)\le L_t(s),\qquad 0<s\le s_0.
 \label{eq:id-transform-lower-bound}
\end{equation}
For every $x>0$,
\[
 \int_0^{s_0}xe^{-sx}\,\dd s\le1,\qquad
 \int_0^{s_0}sx^2e^{-sx}\,\dd s\le1,\qquad
 sx e^{-sx}\le e^{-1}.
\]
Taking absolute values before applying Tonelli therefore yields
\begin{equation}
 \begin{aligned}
 \int_0^{s_0}|A_t(s)|\,\dd s
       &\le D^{-1}\E_t|\log X|,\\
 \int_0^{s_0}s\E_{t,s}[X^2|\log X|]\,\dd s
       &\le D^{-1}\E_t|\log X|,\\
 s\E_{t,s}X&\le(eD)^{-1},\\
 \int_0^T\int_0^{s_0}|h_t(s)|\,\dd s\,\dd t
       &\le T\left(\frac2D+\frac1{eD^2}\right)C_{\log,T}<\infty.
 \end{aligned}
 \label{eq:id-absolute-fubini}
\end{equation}
The last bound includes the product term in \eqref{eq:id-tangent-expanded}; it follows by multiplying the first bound by $(eD)^{-1}$.
Together with the joint measurability just proved, this establishes that $h(t,s):=h_t(s)$ belongs to $L^1([0,T]\times(0,s_0))$.
Fubini's theorem therefore applies on every subrectangle $[0,t]\times(0,s)$ with $t\le T$ and $s\le s_0$.

For every fixed $t$,
\begin{equation}
 \lim_{s\downarrow0}sA_t(s)=0,
 \qquad |sA_t(s)|\le\frac{\E_t|\log X|}{eD}\quad(0<s\le s_0).
 \label{eq:id-zero-endpoint}
\end{equation}
Indeed, $sx e^{-sx}\log x\to0$ pointwise, and its absolute value is at most $|\log x|/e$.
Dominated convergence under $\mu_t$, followed by division by $L_t(s)\ge D$, proves the limit.
The bound in \eqref{eq:id-log-moment} also permits passage of this limit through a time integral.
The pointwise limit in $t$ and this integrable bound suffice for that passage.

The fundamental theorem of calculus applied to \eqref{eq:id-current-tangent} on $[\varepsilon,s]$, followed by \eqref{eq:id-zero-endpoint}, gives
\[
 \int_0^s h_t(r)\,\dd r=sA_t(s).
\]
The integral is absolutely convergent, and \eqref{eq:id-absolute-fubini} justifies its time--Laplace Fubini exchange.
The logarithmic moment in \eqref{eq:id-log-moment} has supplied all the required endpoint control.

\medskip\noindent\emph{The Laplace equation.}
By Tonelli's theorem, $\int_0^s g_t(r)\,\dd r=\Psi_t(s)$.
Integrating \eqref{eq:id-resolvent-evolution}, with the absolute bounds just proved, therefore yields
\begin{equation}
 \Psi_t(s)-\Psi_0(s)=\int_0^t sA_u(s)\,\dd u.
 \label{eq:id-exponent-evolution}
\end{equation}
The normalization is fixed by $\Psi_t(0)=0$ and the vanishing boundary term in \eqref{eq:id-zero-endpoint}.
Applying the absolutely continuous chain rule to $L_t=e^{-\Psi_t}$ gives
\begin{equation}
 L_t(s)-L_0(s)
 =-\int_0^t s\int_{(0,\infty)}x\log x\,e^{-sx}\,\mu_u(\dd x)\,\dd u,
 \qquad s>0.
 \label{eq:id-laplace-evolution}
\end{equation}
Here $L_u(s)A_u(s)=\E_u[X\log X e^{-sX}]$ has removed the tilt normalization.
For each fixed $s>0$, the function $x\log x\,e^{-sx}$ is bounded and continuous and tends to zero at both endpoints.
Narrow continuity of $(\mu_u)$ makes the time integrand in \eqref{eq:id-laplace-evolution} continuous.

\medskip\noindent\emph{The weak equation on value space.}
We next derive the weak equation for the value distributions.
We approximate a test and its derivative simultaneously, so that the transport term converges as well.
Let $H\in C_c^1((0,\infty))$.
On $0<u<1$, set $J(u)=H(-\log u)$.
Because $H$ vanishes near zero and infinity, $J$ extends by zero to a $C^1$ function on $[0,1]$, vanishing near both endpoints.
There are polynomials $p_n$ satisfying
\begin{equation}
 \|p_n-J\|_{\infty,[0,1]}+
 \|p_n'-J'\|_{\infty,[0,1]}\longrightarrow0.
 \label{eq:id-c1-polynomials}
\end{equation}
We construct these polynomials as follows.
For each integer $n\ge1$, form the Bernstein polynomial of $J'$,
\[
 q_n(u)=\sum_{j=0}^n J'(j/n)\binom{n}{j}u^j(1-u)^{n-j}.
\]
For a binomial random variable $N$ with parameters $(n,u)$, their error is bounded by $\E|J'(N/n)-J'(u)|$.
On $|N/n-u|\le\delta$ use the modulus of continuity of $J'$; on its complement use $2\|J'\|_\infty$ and
\[
 \mathbb P\{|N/n-u|>\delta\}
 \le\frac{u(1-u)}{n\delta^2}\le\frac1{4n\delta^2}.
\]
First choosing small $\delta$ and then large $n$ proves uniform convergence of $q_n$ to $J'$.
The polynomials $p_n(u)=J(0)+\int_0^u q_n(v)\,\dd v$ prove \eqref{eq:id-c1-polynomials}.
The same Bernstein estimate applied to any continuous function, without differentiating it, proves the uniform polynomial approximation used above for transform uniqueness.

Put $H_n(x)=p_n(e^{-x})$ and define the value-space transport operator by
\[
 \mathscr{B}H(x)=x\log x\,H'(x),\qquad x>0.
\]
In the following estimates, norms of $H_n,H$ are taken on $(0,\infty)$, whereas norms of $p_n,J$ and their derivatives are taken on $[0,1]$:
\begin{equation}
 \begin{aligned}
 \|H_n-H\|_\infty&\le\|p_n-J\|_\infty,\\
 \|\mathscr{B}H_n-\mathscr{B}H\|_\infty
 &\le\left(\sup_{x>0}|x\log x|e^{-x}\right)
                   \|p_n'-J'\|_\infty\longrightarrow0.
 \end{aligned}
 \label{eq:id-generator-norm}
\end{equation}
The displayed supremum is finite, because its integrand is continuous and vanishes at zero and infinity.
Each $H_n$ is a finite linear combination of a constant and the functions $e^{-jx}$ for positive integers $j$.
Its weak evolution is therefore supplied by \eqref{eq:id-laplace-evolution}; the constant has zero evolution since each $\mu_u$ is a probability.
Both uniform limits in \eqref{eq:id-generator-norm} can be passed through probability integrals and the finite time integral.
We obtain
\begin{equation}
 \mu_t(H)-\mu_0(H)
 =\int_0^t\mu_u(x\log x\,H'(x))\,\dd u,
 \qquad H\in C_c^1((0,\infty)).
 \label{eq:id-value-weak}
\end{equation}
The simultaneous convergence in \eqref{eq:id-generator-norm} is what transfers the Laplace equation to the transport operator.

\medskip\noindent\emph{Transport of logarithmic values.}
The gamma--Dirichlet representation shows that $\mu_t$ gives full mass to $(0,\infty)$.
Let $\lambda_t=(\log)_*\mu_t$ be the push-forward of $\mu_t$ under $x\mapsto\log x$, so $\lambda_t=\Law(\log X_t)$.
Thus $\lambda_t$ describes logarithmic values, while $F_t$ describes logarithmic rates.
It is narrowly continuous, since logarithm is a continuous map on $(0,\infty)$, and \eqref{eq:id-log-moment} gives a uniform first moment of $\lambda_t$.

For $\zeta\in C_c^1(\R)$, the function $H(x)=\zeta(\log x)$ lies in $C_c^1((0,\infty))$.
Substitution into \eqref{eq:id-value-weak} gives
\begin{equation}
 \lambda_t(\zeta)-\lambda_0(\zeta)
 =\int_0^t\lambda_u(z\zeta'(z))\,\dd u.
 \label{eq:id-log-transport}
\end{equation}
The characteristics of this equation are dilations.
We verify the resulting identity directly with backward test functions.
Fix $0<t\le T$ and $\zeta\in C_c^\infty(\R)$, and set
\begin{equation}
 \zeta_u(z)=\zeta(e^{t-u}z),\qquad 0\le u\le t.
 \quad\text{Then}\quad
 \partial_u\zeta_u+z\partial_z\zeta_u=0.
 \label{eq:id-backward-test}
\end{equation}
All these tests and their derivatives have support in one compact interval, and their time and spatial derivatives are jointly continuous.

To use these time-dependent tests in the fixed-test identity, we write the increments on a partition.
Take a partition $0=t_0<\cdots<t_N=t$ and decompose each increment as
\[
 \begin{aligned}
 &\lambda_{t_{i+1}}(\zeta_{t_{i+1}})
       -\lambda_{t_i}(\zeta_{t_i})\\
 &\quad=\lambda_{t_{i+1}}(\zeta_{t_{i+1}}-\zeta_{t_i})
       +\lambda_{t_{i+1}}(\zeta_{t_i})
       -\lambda_{t_i}(\zeta_{t_i})\\
 &\quad=\int_{t_i}^{t_{i+1}}
          \lambda_{t_{i+1}}(\partial_u\zeta_u)\,\dd u
       +\int_{t_i}^{t_{i+1}}
          \lambda_u(z\partial_z\zeta_{t_i})\,\dd u,
 \end{aligned}
\]
where the second integral uses the fixed-test identity \eqref{eq:id-log-transport}.
Joint continuity, common compact support, and narrow continuity of $\lambda$ imply joint continuity of $(r,u)\mapsto\lambda_r(\partial_u\zeta_u)$ and of the analogous spatial-derivative pairing.
For example, joint continuity follows by adding a supremum-norm difference of the test functions to the weak continuity term for one fixed test.
On the compact time square it is uniform.
Consequently, as the mesh tends to zero, the summed integrals converge to
\[
 \int_0^t\lambda_u(\partial_u\zeta_u+z\partial_z\zeta_u)\,\dd u=0.
\]
The telescoped left side is $\lambda_t(\zeta)-\lambda_0(\zeta(e^t\,\cdot))$.
Thus
\[
 \lambda_t(\zeta)=\lambda_0(\zeta(e^t\,\cdot)),
 \qquad \zeta\in C_c^\infty(\R).
\]
Compactly supported smooth tests determine finite Borel measures on $\R$, as follows by approximating compactly supported continuous functions uniformly by smooth ones and then approximating interval indicators.
Therefore
\begin{equation}
 \lambda_t=(z\mapsto e^t z)_*\lambda_0,
 \qquad
 \mu_t=(x\mapsto x^{e^t})_*\mu_0=\Law(X_0^{e^t}).
 \label{eq:id-transport-identification}
\end{equation}
The value laws are supported on $(0,\infty)$, so exponentiating the identified logarithmic laws proves \eqref{eq:id-conclusion} for every $t\in[0,T]$.
\end{proof}

\section{Proof of the main theorem}
\label{sec:completion}

\begin{proof}[Proof of Theorem~\ref{thm:main}]
Consider first the finite gamma convolution \eqref{eq:finite-input}, and fix a real $q>1$.
Its log-rate probability $F_0$ has finite support and hence belongs to $\Ptwo$.
We apply Theorem~\ref{thm:evolution} on the finite interval $[0,T]$ with $T=\log q$.
It gives a narrowly continuous probability curve $F_t$ with uniformly bounded second log-rate moments, satisfying the exact weak generator equation and $B_t=B_0e^{-t}$.
By \eqref{eq:log-moment-admissibility}, each $U_t=B_t\exp_*F_t$ is a positive admissible Thorin measure, so it defines a zero-drift GGC law $\mu_t$.

The initial law is $\mu_0=\Law(X)$.
The weak equation and the uniform second-moment bound supplied by the construction are precisely the hypotheses of Theorem~\ref{thm:identification}.
It follows that
\[
 \mu_t=\Law(X^{e^t})\quad(0\leq t\leq T).
\]
In particular $\Law(X^q)\in\GGC$.
This proves the assertion for every finite gamma convolution and every real $q>1$; the case $q=1$ is immediate.

Now consider an arbitrary nonnegative GGC random variable $X$ and fix any $q\geq1$.
By Lemma~\ref{lem:ggc-closure} there are finite gamma convolutions $X_m\Rightarrow X$.
Each $X_m^q$ is GGC by the first part.
The function $x\mapsto x^q$ is continuous on $[0,\infty)$, so $X_m^q\Rightarrow X^q$.
The conclusion follows from weak closure of the GGC class.
\end{proof}

\begin{remark}
The approximation in Lemma~\ref{lem:ggc-closure} includes positive drift, infinite Thorin mass, and the constant zero.
The second logarithmic moment is used only to construct the evolution for a finite gamma convolution; it imposes no restriction on the limiting GGC variable.
The constants may depend on the approximating variable, since weak closure is applied after each powered approximant has been shown to belong to $\GGC$.
Likewise, a fixed real $q>1$ requires only the finite interval $[0,\log q]$.
\end{remark}

\appendix
\section{Measurable realizations}
\label{sec:measurable-realizations}

We give the parameterized Dirichlet construction and the measurable version of the bounded phase used above.
The common probability space makes the coefficients jointly measurable in the mass, the base probability, and the distinguished rate.

\subsection{Dirichlet probabilities with varying parameters}

We use the stick-breaking formula of Lemma~\ref{lem:stick-breaking}, choosing the locations by quantiles and the break fractions by common uniform random variables.

\begin{lemma}[A parameterized Dirichlet realization]
\label{lem:parameterized-dirichlet}
There is a fixed probability space \((\Omega,\mathcal A,\mathbb P)\) and jointly Borel maps
\[
 (B,F,\omega)\longmapsto Q(B,F,\omega),\qquad
 (B,F,y,\omega)\longmapsto\widehat P(B,F,y,\omega)
\]
into \(\mathcal P(\R)\), for \(B>0\), \(F\in\mathcal P(\R)\), and \(y\in\R\), such that
\[
 Q(B,F,\cdot)\sim\DP(BF),\qquad
 \widehat P(B,F,y,\cdot)\sim\DP(BF+\delta_y).
\]
The construction uses a single event of probability one on which the stick-breaking weights sum to one for every \(B>0\) and every \(F\).
All spaces of probabilities carry their narrow Borel sigma-fields.
\end{lemma}

\begin{proof}
For \(F\in\mathcal P(\R)\) and \(0<u<1\), define its quantile by
\[
 q_F(u)=\inf\{x\in\R:F((-\infty,x])\ge u\}.
\]
This is a finite real number.
For every \(a\in\R\),
\begin{equation}\label{eq:quantile-joint-borel}
 \{(F,u):q_F(u)<a\}
 =\bigcup_{\substack{r\in\mathbb Q\\r<a}}
       \{(F,u):F((-\infty,r])\ge u\}.
\end{equation}
The evaluation \(F\mapsto F((-\infty,r])\) is Borel: the bounded continuous functions \((1-n(x-r)_+)_+\) decrease to the indicator of this half-line, so their integrals converge to that evaluation.
Thus \((F,u)\mapsto q_F(u)\) is jointly Borel.
Right continuity of distribution functions also gives \(\mathbb P\{q_F(S)\le x\}=F((-\infty,x])\) for a uniform \(S\) on \((0,1)\).

Take \(\Omega=(0,1)^{\mathbb N}\times(0,1)^{\mathbb N}\times(0,1)\) with its product Borel sigma-field and product uniform probability.
Write \((S_j)_j,(T_j)_j,T_0\) for its independent coordinate variables.
Suppressing \(\omega\) in the formulas, put
\begin{equation}\label{eq:parameterized-stick-breaking}
 \begin{gathered}
 Z_j(F)=q_F(S_j),\qquad V_j(B)=1-T_j^{1/B},\\
 W_j(B)=V_j(B)\prod_{i<j}(1-V_i(B)).
 \end{gathered}
\end{equation}
These are jointly Borel in their parameters and \(\omega\).
The decreasing products \(\prod_{j\le m}T_j\) tend to zero almost surely: their expectations are \(2^{-m}\), so their nonnegative limit has expectation zero.
Denote this Borel event by \(\Omega_0\).
It does not depend on \(B,F\), and on it
\[
 R_m(B):=\prod_{j\le m}(1-V_j(B))
        =\left(\prod_{j\le m}T_j\right)^{1/B}
        \longrightarrow0\qquad\text{for every }B>0.
\]
In particular \(\sum_jW_j(B)=1\) simultaneously for all parameters on \(\Omega_0\); no intersection of parameter-dependent null sets is needed.

The probability-valued maps
\[
 Q_m(B,F)=\sum_{j\le m}W_j(B)\delta_{Z_j(F)}+R_m(B)\delta_0
\]
are jointly Borel, since finite convex combinations and \(z\mapsto\delta_z\) are continuous in the narrow topology.
On \(\Omega_0\) they converge narrowly to \(Q(B,F)=\sum_{j\ge1}W_j(B)\delta_{Z_j(F)}\); the tail error against a bounded continuous test is at most twice its supremum norm times \(R_m(B)\).
Set both \(Q_m\) and \(Q\) equal to \(\delta_0\) on \(\Omega_0^c\).
The resulting everywhere pointwise limit is a jointly Borel map into the metrizable space \(\mathcal P(\R)\).
For each fixed \(B,F\), its law is \(\DP(BF)\) by Lemma~\ref{lem:stick-breaking}.

Finally, set \(Z_*(B)=1-T_0^{1/B}\) and
\[
 \widehat P(B,F,y)=(1-Z_*(B))Q(B,F)+Z_*(B)\delta_y.
\]
This map is jointly Borel.
For each fixed \(B,F\), the variable \(Z_*(B)\) is \(\operatorname{Beta}(1,B)\) and independent of \(Q(B,F)\).
Lemma~\ref{lem:dirichlet-posterior}, in log-rate coordinates, therefore gives the stated posterior law.
Passing through \(\exp_*\) supplies the corresponding jointly Borel realizations on positive rates as well.
\end{proof}

\subsection{Proof of Lemma~\ref{lem:phase}}
\label{sec:phase-proof}

\begin{proof}[Proof of Lemma~\ref{lem:phase}]
The function \(M_P\) is a nonzero Stieltjes function, since \(\int(1+b)^{-1}P(\dd b)\le1\).
For \(z\) in the upper half-plane,
\[
  \operatorname{Im}M_P(z)
  =-\operatorname{Im}z\int|z+b|^{-2}P(\dd b)<0.
\]
Thus the arguments in \eqref{eq:phase-representative} are unambiguous principal arguments, with \(-\pi<\arg M_P(z)<0\).
Apply \eqref{eq:ssv-exponential} to \(f=1/M_P\), and subtract its real logarithms at \(s\) and \(1\).
Negating the resulting identity gives exactly the sign in \eqref{eq:phase-anchor}.
The uniqueness assertion follows from uniqueness in \eqref{eq:ssv-exponential}: the value at \(1\) fixes the otherwise free multiplicative constant.

Let \(\xi_P^0\) be any almost-everywhere phase initially provided by that representation.
Analytic continuation of \eqref{eq:ssv-exponential} and its imaginary part give
\[
  \arg(1/M_P(-t+i\varepsilon))
  =\int_0^\infty
       \frac{\varepsilon\,\xi_P^0(u)}
            {(u-t)^2+\varepsilon^2}\,\dd u,
  \qquad \varepsilon>0.
\]
Extend \(\xi_P^0\) by zero to the negative half-line.
The normalized integral is its Poisson approximate identity; at each Lebesgue point it tends to \(\xi_P^0(t)\).
For completeness, this convergence follows by splitting the integral into a neighborhood of the point, where the Lebesgue-point averages of the error tend to zero, and its complement, whose Poisson mass tends to zero.
Boundedness of \(\xi_P^0\) controls the complement.
Hence \eqref{eq:phase-representative} agrees with \(\xi_P^0\) almost everywhere and preserves \eqref{eq:phase-anchor}.

For fixed \(n\), the map \((P,t)\mapsto M_P(-t+i/n)\) is jointly Borel (in fact continuous).
The kernel is bounded by \(n\), is continuous in the rate variable, and its dependence on \(t\) is locally uniform, with derivative bounded by \(n^2\).
Taking a continuous argument on the strict lower half-plane and then a limsup proves joint Borel measurability of the specified representative.
The minus sign remains inside the limsup in \eqref{eq:phase-representative}, so the definition is unambiguous also at exceptional boundary points.

Finally,
\[
  \int_0^\infty
   \left|\frac1{s+t}-\frac1{1+t}\right|\dd t=|\log s|,
\]
and, for every integer \(k\ge1\),
\[
  \partial_s^k\log M_P(s)
  =(-1)^k k!\int_0^\infty
                    \frac{\xi_P(t)}{(s+t)^{k+1}}\,\dd t,
  \qquad
  k!\int_0^\infty(s+t)^{-k-1}\dd t
       =(k-1)!s^{-k}.
\]
These estimates prove the absolute convergence and justify all the stated differentiations.
\end{proof}

\bibliographystyle{plainnat}
\bibliography{references}
\end{document}